\documentclass{amsart}
\usepackage{latexsym,amsxtra,amscd,ifthen,amsmath,color,url}
\usepackage[normalem]{ulem}
\usepackage[usenames,dvipsnames]{xcolor} \usepackage{amsfonts}
\usepackage{verbatim} \usepackage{amsmath} \usepackage{amsthm}
\usepackage{amssymb}
\usepackage[linkcolor=red,citecolor=blue,urlcolor=blue]{hyperref}
\usepackage[notcite,notref, final]{showkeys}

\usepackage{pgf} \usepackage{tikz} 
\usetikzlibrary{shapes,shapes.geometric,arrows,fit,calc,positioning,automata,}
\usepackage[latin1]{inputenc} \usepackage{verbatim}

\usepackage[all]{xy}
\input xy
\xyoption{all}

\numberwithin{equation}{section} \theoremstyle{plain}
\newtheorem{theorem}{Theorem}[section]
\newtheorem{lemma}[theorem]{Lemma}

\newtheorem{proposition}[theorem]{Proposition}
\newtheorem{corollary}[theorem]{Corollary}
\newtheorem{conjecture}[theorem]{Conjecture}
\newtheorem{definitionlemma}[theorem]{Definition-Lemma}
 \theoremstyle{definition}
\newtheorem{definition}[theorem]{Definition}

\newtheorem{remark}[theorem]{Remark}
\newtheorem{question}[theorem]{Question}
\newtheorem{notation}[theorem]{Notation}

\mathchardef\mhyphen="2D

\makeatletter              
\let\c@equation\c@theorem  
\makeatother

\DeclareMathOperator{\GKdim}{GKdim}

\DeclareMathOperator{\Ext}{Ext}

 \DeclareMathOperator{\Spec}{Spec}

\DeclareMathOperator{\Hom}{Hom}

\DeclareMathOperator{\depth}{depth}

\newcommand{\CC}{\mathcal C}
\newcommand{\DD}{\mathcal D}

\newcommand{\EE}{\mathcal E}

\newcommand{\kk}{\Bbbk}

\begin{document}

\title[On the quadratic dual of the Fomin-Kirillov algebras]
{On the quadratic dual of the \\Fomin-Kirillov algebras}
\author{Chelsea Walton and James J. Zhang}

\address{Walton: Department of Mathematics, 
The University of Illinois at Urbana-Champaign, 
Urbana, IL 61801 USA}

\email{notlaw@illinois.edu}

\address{Zhang: Department of Mathematics, Box 354350, University of
Washington, Seattle, WA 98195, USA}

\email{zhang@math.washington.edu}

\begin{abstract} 
We study ring-theoretic and homological properties of the quadratic 
dual (or Koszul dual) $\mathcal{E}_n^!$ of the Fomin-Kirillov 
algebras $\mathcal{E}_n$; these algebras are connected 
$\mathbb{N}$-graded and  are defined for  $n \geq 2$. We 
establish that the algebra $\mathcal{E}_n^!$ is module-finite over its center (so, satisfies a polynomial identity), 
is Noetherian, and has Gelfand-Kirillov 
dimension $\lfloor n/2 \rfloor$ for each $n \geq 2$. We also 
observe that $\mathcal{E}_n^!$ is not prime  for  $n \geq 3$. 
By a result of Roos, $\mathcal{E}_n$ is not Koszul for 
$n \geq 3$, so neither is $\mathcal{E}_n^!$ for $n \geq 3$. 
Nevertheless, we prove that $\mathcal{E}_n^!$ is 
Artin-Schelter (AS-)regular if and only if $n=2$, and that 
$\mathcal{E}_n^!$  is both AS-Gorenstein and AS-Cohen-Macaulay 
if and only if $n=2,3$.  We also show that the depth of 
$\mathcal{E}_n^!$ is $\leq 1$ for each $n \geq 2$, 
conjecture we have equality, and show 
this claim holds for $n =2,3$. Several other directions for 
further examination of $\mathcal{E}_n^!$ are suggested at the 
end of this article.
\end{abstract} 


\maketitle

\setcounter{section}{-1}


\section{Introduction} 
\label{xxsec0}

Throughout this work, we let $\kk$ denote an algebraically closed 
field of characteristic 
zero, and we consider all algebraic structures to be $\kk$-linear.

\smallskip
In 1999, Sergey Fomin and Anatol Kirillov introduced a family of quadratic 
algebras for the study of the quantum and the ordinary cohomology of flag 
manifolds in \cite{FK}. Since then these algebras, now referred to as
{\it Fomin-Kirillov algebras}~$\EE_n$ [Definition~\ref{xxdef0.1}], have been studied 
extensively by many authors with connections to algebraic combinatorics \cite{BLM,GR,MPP,Po},  algebraic geometry and cohomology theory \cite{KM,Len}, Hopf algebras and Nichols algebras \cite{AG, FP, MS},  noncommutative geometry \cite{Ma1,Ma2}, and number theory \cite{MT},
 and generalizations of these algebras have been examined in several works \cite{Baz,La,MS}. The cohomology algebras  $\Ext^{\ast}_{\EE_n}(\Bbbk,\Bbbk)$ of the Fomin-Kirillov algebras have also been of interest: They are related to the cohomology ring of quantum shuffle algebras in \cite{ETW} for example, and is studied in detail by \c{S}tefan and Vay \cite{SV} when $n=3$.  Moreover, several questions asked by Fomin and Kirillov in their 1999 work are still open-- most notably, it is unknown whether $\EE_n$ is finite dimensional for 
$n\geq 6$ \cite[Problem~2.2]{FK}. (The algebra $\EE_n$ is finite-dimensional  when $n =2,3,4,5$ \cite[~(2.8)]{FK}.) Many of these questions can be recast in terms of  $\Ext^{\ast}_{\EE_n}(\Bbbk,\Bbbk)$-- for instance, the regularity of $\Ext^{\ast}_{\EE_n}(\Bbbk,\Bbbk)$  as an $A_{\infty}$-algebra is in a certain sense equivalent to the finite-dimensionality of $\EE_n$ \cite{LPWZ}. 

\smallskip
With the aim of providing insight into the structure of both $\EE_n$ 
and its cohomology algebra we examine in this article the subalgebra 
of $\Ext^{\ast}_{\EE_n}(\Bbbk,\Bbbk)$ generated by 
$\Ext^1_{\EE_n}(\Bbbk,\Bbbk)$. Namely, we study the structure of 
the {\it quadratic dual} (or {\it Koszul dual}) $\EE_n^!$ of 
$\EE_n$, also known as the {\it diagonal subalgebra} 
$\bigoplus \Ext^{i,i}_{\EE_n}(\Bbbk,\Bbbk)$ of $\Ext^{\ast}_{\EE_n}(\Bbbk,\Bbbk)$. 

\bigskip
\noindent {\bf Note}. Due to the scope of fields in which the 
Fomin-Kirillov algebras appear, the exposition of this article 
is intended for an audience broader than experts in 
noncommutative graded algebras. As discussed below, much of 
the theory of commutative local algebras has been generalized 
to the noncommutative setting by working with algebras that 
are connected graded. An $\mathbb{N}$-graded $\kk$-algebra 
$A = \bigoplus_{i \in \mathbb{N}}A_i$ is, 
by definition, {\it connected} if $A_0 = \kk$.
\medskip

\smallskip
To begin, the presentations of $\EE_n$ and of its quadratic 
dual $\EE_n^!$ are provided below.

\begin{definition}[$\EE_n$, $x_{i,j}$]  \cite[Definition 2.1]{FK}
\label{xxdef0.1}
For $n\geq 2$, the {\it Fomin-Kirillov algebra},
 $\EE_n$, is an associative $\kk$-algebra generated by 
indeterminates $\{x_{i,j}\mid 1\leq i<j\leq n\}$ of degree one, subject 
to the following quadratic relations:
$$
\begin{array}{ll}
x_{i,j}^2=0 
& \quad i<j,\\
x_{i,j}x_{j,k}-x_{j,k}x_{i,k}-x_{i,k}x_{i,j}=0 
& \quad i<j<k,\\
x_{j,k}x_{i,j}-x_{i,k}x_{j,k}-x_{i,j}x_{i,k}=0 
&\quad  i<j<k,\\
x_{i,j}x_{k,l}-x_{k,l}x_{i,j}=0 
& \quad \{i,j\}\cap \{k,l\}=\emptyset,~ i<j, ~k<l.
\end{array}
$$
\end{definition}

The presentation of $\EE_n^{!}$ is obtained in a 
straight-forward manner from the definition above 
and is recorded below. For a graded quadratic 
$\kk$-algebra $A = T(V)/(R)$ with $T(V)$ the tensor 
algebra on a $\kk$-vector space $V$ and $(R)$ the 
two-sided ideal generated by a space $R \subset V\otimes V$,  
the {\it quadratic dual} $A^!$ of $A$ is $T(V^*)/(R^\perp)$ 
where $V^*$ is the linear dual of $V$ and $R^\perp$ is the 
orthogonal complement of $R$ in $V \otimes V$.

\begin{definitionlemma}[$\EE_n^!$, $y_{i,j}$] \label{xxdeflem0.2} 
The quadratic dual $\EE_n^{!}$ of $\EE_n$ is an associative 
$\kk$-algebra generated by the indeterminates
$\{y_{i,j}:=x_{i,j}^{\ast}\mid 1\leq i<j\leq n\}$ of degree one, 
subject to the following  quadratic relations:
$$\begin{array}{ll}
y_{i,j}y_{j,k}+y_{j,k}y_{i,k}=0 \qquad & \text{$i,j,k$ distinct, 
where $y_{j,i}=-y_{i,j}$ for $i<j$};\\
y_{i,j}y_{k,l}+y_{k,l}y_{i,j}=0 \qquad 
& \{i,j\}\cap \{k,l\}=\emptyset,~ i<j, ~ k<l.
\end{array} 
$$

\vspace{-.2in}

 \qed
\end{definitionlemma}

Our first result pertains to ring-theoretic results on  $\EE_n^{!}$: 
one on the {\it polynomial identity (PI)} property 
\textnormal{[Definition~\ref{xxdef1.6}, Remark~\ref{xxrem1.7}]}, 
i.e., a condition for a ring being ``measurably" close to being 
commutative (e.g., via its {\it PI degree}); and another on the 
{\it Gelfand-Kirillov (GK) dimension} 
\textnormal{[Definition~\ref{xxdef1.1}, Remark~\ref{xxrem1.2}]}, 
i.e., a growth measure on a ring that serves as a noncommutative 
version of Krull dimension.

\begin{theorem}
\label{xxthm0.3} 
The algebra $\EE_n^{!}$ satisfies the following properties.
\begin{enumerate}
\item[(1)]  $\EE_n^{!}$ is Noetherian and 
is module-finite over its center (so, satisfies a polynomial 
identity) when $n \geq 2$. 
\item[(2)]  $\EE_n^{!}$ has Gelfand-Kirillov dimension 
$\lfloor n/2\rfloor$ when $n \geq 2$.
\item[(3)] $\EE_n^{!}$ is not prime when $n \geq 3$.
\end{enumerate}
\end{theorem}

For Theorem~\ref{xxthm0.3}(3), recall that a ring $A$ is {\it prime} if for all 
nonzero elements $a,b$ we get $aAb \neq 0$ (a weaker version of the domain property), 
and a ring is {\it semiprime} if it contains no nilpotent ideals 
(a weaker version of the prime property). Semiprime rings are still quite 
useful in their own right (see, e.g., \cite[Chapter~6]{GW}). So, we ask:

\begin{question} \label{xxque0.4}
(1) Are the algebras $\EE_n^{!}$ semiprime?

(2) What is the PI degree of $\EE_n^{!}$?
\end{question}

In any case, we obtain the following immediate results on the cohomology algebra  
$\Ext^{\ast}_{\EE_n}(\Bbbk,\Bbbk)$  since $\EE_n^!$ is a subalgebra of  
$\Ext^{\ast}_{\EE_n}(\Bbbk,\Bbbk)$.

\begin{corollary} 
\label{xxcor0.5} 
We have that $\GKdim(\Ext^{\ast}_{\EE_n}(\Bbbk,\Bbbk))$ 
$\geq \lfloor n/2\rfloor$ when $n \geq 2$, and  
that $\Ext^{\ast}_{\EE_n}(\Bbbk,\Bbbk)$ is not a domain 
when $n \geq 3$. \qed
\end{corollary}

The GKdim bound in Corollary~\ref{xxcor0.5} is not sharp: 
Indeed, $\GKdim \Ext^{\ast}_{\EE_3}(\kk,\kk)=2$ by a 
result of Stefan and Vay \cite[Theorem 4.17]{SV}. 
Nonetheless, it was shown recently by Ellenberg, Tran, and 
Westerland that the growth of $\EE_n^!$ and of 
$\Ext^{\ast}_{\EE_n}(\kk,\kk)$ are related to the growth 
of the homology of Hurwitz spaces \cite{ETW}.

\smallskip

Now we consider various homological conditions on  $\EE_n^{!}$. 
To start,  note that $\EE_{n}$ is not Koszul for $n\geq 3$, 
due to a result of Roos \cite{Ro}; thus, $\EE_{n}^!$ is not 
Koszul for $n\geq 3$. Thus, $\EE_n^!$ is a proper subalgebra 
of $\Ext^{\ast}_{\EE_n}(\Bbbk,\Bbbk)$ (see \cite[Section~1.3]{PP}). 
In any case, we examine $\EE_n^{!}$ in view of the hierarchy 
of {\it Artin-Schelter (AS)} versions of desirable homological 
properties for connected graded algebras that are not 
necessarily commutative; see, e.g., \cite[Introduction]{KKZ} 
for more details. In short, one has:
\bigskip

\begin{center}
\begin{tabular}{rl}
\hspace{-.15in}
{\it AS-regular} [Definition~\ref{xxdef3.1}(2)]  
& $\Longrightarrow$ {\it AS-Gorenstein} [Definition~\ref{xxdef3.1}(1)] \\
& $\Longrightarrow$ {\it AS-Cohen-Macaulay} [Definition~\ref{xxdef3.1}(4)].
\end{tabular}
\end{center}

\bigskip

A version of the {\it classical complete intersection}  
condition [Definition~\ref{xxdef3.2}(3)] is also related; 
see Figure~\ref{fig:hom} in Section~\ref{xxsec3} for more details. 
Our second result determines whether $\EE_n^!$ satisfies the conditions above.

\begin{theorem}
\label{xxthm0.6} We have the following statements for $\EE_n^!$ 
\textnormal{(}and for $\EE_2$, $\EE_3$\textnormal{)}.
\begin{enumerate}
\item[(1)]
$\EE_{n}^{!}$ is Artin-Schelter-regular if and only if $n=2$.
\item[(2)]
$\EE_{n}^{!}$ is Artin-Schelter-Gorenstein if and only if $n=2,3$.
\item[(3)]
$\EE_{n}^{!}$ is Artin-Schelter-Cohen-Macaulay if and only if $n=2,3$.
\item[(4)] $\EE_2^!$, $\EE_3^!$, $\EE_2$, and $\EE_3$ are classical 
completion intersections.
\end{enumerate}
\end{theorem}

Since $\EE_n^!$ is a Noetherian algebra that satisfies a polynomial 
identity [Theorem~\ref{xxthm0.3}], one can measure the failure of 
the AS Cohen-Macaulay property of $\EE_n^!$ by considering its 
{\it depth}  [Definition~\ref{xxdef3.1}(11)]. Namely, the depth 
of a connected graded, Noetherian algebra $A$ that satisfies a 
polynomial identity is  bounded above by its Gelfand-Kirillov 
dimension, and we have equality of these values if and only if 
$A$ is AS-Cohen-Macaulay, due to results of J{\o}rgensen 
\cite[Theorem~2.2]{Jo2}. In fact, 
we verify that ${\mathcal E}_n^!$ has depth $\leq 1$ 
for $n \geq 2$ [Theorem~\ref{xxthm3.12}] (cf., Theorem~\ref{xxthm0.3}(2)), 
and computational evidence 
yields the following conjecture :

\begin{conjecture}
\label{xxcon0.7}
The algebras $\EE_n^!$ have depth 1 for all $n \geq 2$.
\end{conjecture}

Note that if Question \ref{xxque0.4}(1) has an affirmative answer,
then Conjecture \ref{xxcon0.7} holds by Theorem \ref{xxthm3.12}; 
see Remark~\ref{xxrem4.1}.

\smallskip There is a similar hierarchy of {\it Auslander} versions 
of the regularity [Definition~\ref{xxdef3.1}(8)] and the 
Gorenstein [Definition~\ref{xxdef3.1}(7)] conditions, in 
comparison with a version of the  Cohen-Macaulay
condition [Definition~\ref{xxdef3.1}(9)] considered 
frequently in noncommutative algebra; see Figure~\ref{fig:hom}. 
From Theorem~\ref{xxthm0.6} we obtain the following consequences.

\begin{corollary} 
\label{xxcor0.8}
The algebra $\EE_n^!$ is Auslander-regular if and only 
if $n=2$, and is both Auslander-Gorenstein and 
Cohen-Macaulay if and only if $n=2,3$. 
\end{corollary}

The proof of Theorem~\ref{xxthm0.3} is provided in Section~\ref{xxsec1}, 
and Theorem~\ref{xxthm0.6} and Corollary~\ref{xxcor0.8} are established 
in Section~\ref{xxsec3}. Two important commutative subalgebras 
${\mathcal C}_n$ and ${\mathcal D}_n$ of ${\mathcal E}_n^!$ 
[Notation~\ref{xxnot1.5}] that play a key role in Section~\ref{xxsec1} 
are examined further in Section~\ref{xxsec2}. 
In addition to Question~\ref{xxque0.4} and 
Conjecture~\ref{xxcon0.7}, more questions and directions for further 
investigation are presented in Section~\ref{xxsec4}.


\section{Ring-theoretic preliminaries and Proof of Theorem~\ref{xxthm0.3}} 
\label{xxsec1}

Here, we recall the Gelfand-Kirillov (GK-)dimension and PI 
property of rings, along with providing examples of and remarks 
on these notions. Towards studying the GK-dimension and PI 
property of $\EE_n^!$, commutative subalgebras $\CC_n$ and 
$\DD_n$ of $\EE_n^!$ are constructed below. After preliminary 
results on $\EE_n^!$, and on its relationship to $\CC_n$ and 
$\DD_n$ are verified, we prove Theorem~\ref{xxthm0.3} at the end of the section.

\medskip
To begin, recall that an $\mathbb{N}$-graded $\kk$-algebra 
$A = \bigoplus_{i \in \mathbb{N}}A_i$  is 
{\it connected graded} (c.g.) if $A_0 = \kk$, and  is 
{\it locally-finite} if $\dim_{\kk} A_i < \infty$ for all $i$.  
Moreover, the {\it Hilbert series} of an $\mathbb{N}$-graded, 
locally-finite algebra $A = \bigoplus_{i \in \mathbb{N}} A_i$ 
is by definition, 
$$\textstyle H_A(t) = \sum_{i \in \mathbb{N}} (\dim_\kk A_i) t^i.$$

\smallskip

\begin{definition} \label{xxdef1.1} \cite{KL} \cite[Chapter~8]{MR}
The {\it Gelfand-Kirillov dimension} 
(or {\it GK-dimension}) of a connected $\mathbb{N}$-graded, locally-finite
$\kk$-algebra $A$ is defined to be
$$
\GKdim (A)=\limsup_{n\to\infty} 
\frac{\log(\sum_{i=0}^{n} \dim_{\Bbbk} A_i)}{\log(n)}.
$$
\end{definition}

\smallskip

\begin{remark} 
\label{xxrem1.2} 
\begin{enumerate}
\item[(1)] Suppose $A$ is finitely-generated. Then 
$\GKdim(A) = 0$ if and only if $\dim_\kk A < \infty$.

\smallskip

\item[(2)] 
Even though the definition of GK-dimension is technical, 
the dimension is computable in many cases. For instance, 
the GK-dimension of a polynomial ring $\kk[z_1, \dots, z_m]$ 
is $m$, and so is any noncommutative
$\kk$-algebra with a 
$\kk$-vector space basis of monomials
$$\hspace{.3in} \bigoplus \kk \; z_{i_1}^{e_{i_1}} z_{i_2}^{e_{i_2}} 
\cdots z_{i_t}^{e_{i_t}} z_1^{r_1} z_2^{r_2} \cdots z_m^{r_m},$$
where the sum runs over finitely many values of $e_{i_j}$ and 
$r_j \geq 0$. For $q \in \kk^\times$, the $q$-polynomial ring 
$$\hspace{.3in} \kk_q[z_1, \dots, z_m]
:= \kk\langle z_1, \dots, z_m\rangle/(z_i z_j - qz_j z_i)_{i<j}$$ 
has such a monomial basis (with $e_{i_j} =0$), so its GK-dimension 
is $m$. In fact, {\it polynomial growth} is the same condition 
as finite GK-dimension.

\smallskip

\item[(3)] 
The GK-dimension of a commutative finitely-generated $\kk$-algebra 
is its Krull dimension \cite[Theorem~8.2.14]{MR}.

\smallskip

\item[(4)] If $B$ is either a subalgebra or a homomorphic image 
of a $\kk$-algebra $A$, then we get that GKdim($B$)$\leq$ GKdim($A$) 
\cite[Proposition~8.2.2]{MR}. Moreover, if $A$ is module-finite over 
a $\kk$-subalgebra 
$C$, then GKdim($A$)$=$ GKdim($C$) \cite[Proposition~8.2.9(ii)]{MR}.
\end{enumerate}
\end{remark}

Towards computing the GK-dimension of $\EE_n^!$, we present a 
set of monomials whose $\kk$-span is $\EE_n^!$.

\begin{lemma}
\label{xxlem1.3} 
Order the generators $\{y_{i,j}\}_{1 \leq i < j \leq n}$ of $\EE_n^!$ such that
$y_{i,j}<y_{k,l}$ if $j<l$, or if $j=l$ and $i<k$. That is,
$$y_{1,2} < y_{1,3} < y_{2,3} < y_{1,4} < y_{2,4} < y_{3,4} 
< y_{1,5} < \cdots < y_{n-2,n} < y_{n-1,n}.$$
Then, every element in $\EE_n^{!}$ is a linear 
combination of the monomials of the form 
$$y_{1,2}^{r_{1,2}}y_{1,3}^{r_{1,3}}y_{2,3}^{r_{2,3}}y_{1,4}^{r_{1,4}}\cdots
y_{n-2,n}^{r_{n-2,n}}y_{n-1,n}^{r_{n-1,n}} \quad \text{ for } ~r_{i,j}\geq 0.$$
\end{lemma}

\begin{proof}
First, we show that if $y_{k,l}> y_{i,j}$, then 
$y_{k,l}y_{i,j}$ can be written as a lower order term by using
relations of Definition-Lemma~\ref{xxdeflem0.2}.
 If $\{k,l\}\cap \{i,j\}=\emptyset$, then 
we use the last relation of $\EE_n^{!}$. Otherwise, for $k>j>i$ we have
 that
$$
y_{i,k}y_{i,j}=-y_{i,j}y_{j,k}, \quad \quad
y_{j,k}y_{i,j}=-y_{i,k}y_{j,k}, \quad \quad
y_{j,k}y_{i,k}=-y_{i,j}y_{j,k}.
$$
Now by Bergman's Diamond Lemma \cite{Be}, every element in 
$\EE_n^{!}$ is a linear combination of the monomials 
$y_{1,2}^{r_{1,2}}y_{1,3}^{r_{1,3}}y_{2,3}^{r_{2,3}}y_{1,4}^{r_{1,4}}\cdots
y_{n-2,n}^{r_{n-2,n}}y_{n-1,n}^{r_{n-1,n}}$.
\end{proof}

The following is a preliminary result needed for constructing 
a central subalgebra of $\EE_n^!$ to aid in determining its 
GK-dimension (and PI property, introduced later).

\begin{lemma}
\label{xxlem1.4} Let $a_{i,j}:=y_{i,j}^2$. We have that
$a_{i,j}$ is central in $\EE_{n}^{!}$  for all $i<j$, and that the following relations hold in $\EE_{n}^{!}$:
$$a_{i,j}y_{j,k}=y_{j,k} a_{i,k} \quad \text{and} \quad a_{i,j}a_{j,k} = a_{i,j}a_{i,k} ~\quad \forall \text{ distinct } i,j,k.$$ 
\end{lemma}

\begin{proof} 
 If $\{i,j\}\cap \{s,t\}=\emptyset$, then by Definition-Lemma~\ref{xxdeflem0.2} we get that
$$y_{s,t} a_{i,j}~=~y_{s,t} y_{i,j}^2~=~
-y_{i,j}y_{s,t}y_{i,j}~=~y_{i,j}^2 y_{s,t}~=~a_{i,j}
y_{s,t}.$$
On the other hand, we will show that $a_{i,j}$ and $y_{j,k}$ commute for any $i<j$, with $k = i$ or $j$. Recall that $y_{j,i}=-y_{i,j}$ for all $i<j$. Indeed,
$$a_{i,j} y_{j,k} ~=~ -y_{i,j}^2 y_{k,j} ~=~ y_{i,j} y_{k,i} y_{i,j} ~=~ -y_{j,i} y_{k,i} y_{i,j} ~=~ y_{k,j} y_{j,i} y_{i,j}
~=~y_{j,k} a_{i,j}.$$
Hence, $a_{i,j}$ is central.

By Definition-Lemma~\ref{xxdeflem0.2}, we also obtain that
\[
\begin{array}{llll}\medskip
a_{i,j}y_{j,k}&= -y_{i,j} y_{j,k}y_{i,k}&= y_{j,k}y_{i,k}y_{i,k}&=y_{j,k} a_{i,k}, \quad \text{and}\\
a_{i,j}a_{j,k}&=-y_{i,j} y_{j,k} y_{i,k} y_{j,k}&=y_{i,j} y_{j,k} y_{j,i} y_{i,k}&=-y_{i,j} y_{k,j} y_{j,i} y_{i,k}\\
&=y_{k,i} y_{i,j} y_{j,i} y_{i,k}&=a_{i,j} a_{i,k},
\end{array}
\]
as claimed.
\end{proof}

\begin{notation}[$a_{i,j}$, $\CC_n$, $\DD_n$] \label{xxnot1.5} 
From now on, let $a_{i,j}$ denote the central elements $y_{i,j}^2$ 
(see Lemma~\ref{xxlem1.4}) and we use the notation 
$a_{i,j}$ for $a_{j,i}$ when $i>j$. 

\smallskip Let $\CC_n$ be the commutative subalgebra of $\EE_n^{!}$ generated by 
 $\{a_{i,j}\mid 1\leq i<j\leq n\}$. 
Let $\DD_n$ be the commutative algebra generated by 
$\{a_{i,j}\mid 1\leq i<j\leq n\}$, subject to the relations
$$a_{i,j}a_{j,k}=a_{i,j}a_{i,k}
\quad \forall \text{ distinct } i,j,k.$$ 
By Lemma~\ref{xxlem1.4} there is a natural algebra surjection
from $\DD_n$ to $\CC_n$. Later in Proposition \ref{xxpro2.10}, 
we will show that $\DD_n$ is indeed isomorphic to $\CC_n$.
Note that the relations in the algebra $\DD_n$ are similar to those of the 
algebra $A_n$ that appear in \cite[Section 4.1]{BEER}. 
\end{notation}

We use the commutative subalgebras $\CC_n$ and $\DD_n$ to 
study the GK-dimension and the PI property (defined below) of 
$\EE_n^!$ as follows.

\begin{definition} 
\label{xxdef1.6} \cite[Chapter~13]{MR}
A {\it polynomial identity} for a ring $A$ is a monic 
multilinear polynomial $f \in \mathbb{Z}\langle z_1, \dots, z_d\rangle$ 
so that $f(a_1, \dots, a_d) =0$ for all $a_i \in A$. A ring $A$ is 
called a {\it polynomial identity ring}, or is called {\it PI} 
for short, if such a polynomial $f$ exists for $A$.

\smallskip If a PI ring $A$ is prime, then the {\it PI degree} of $A$ is 
the half of the minimal degree of a polynomial identity for $A$.
\end{definition}

\begin{remark} \label{xxrem1.7}
\begin{enumerate}
\item[(1)] The PI property is preserved taking under subalgebra 
and homomorphic image, and algebras that module-finite over a 
commutative subalgebra are PI \cite[Lemma~13.1.7, Corollary~13.1.13]{MR}.

\smallskip

\item[(2)] 
If $A$ is a commutative domain (so, prime), then its PI 
degree is 1 since the commutator is a polynomial identity 
of minimal degree 2. On the other hand, the $q$-polynomial 
ring $\kk_q[z_1, \dots, z_m]$ from Remark~\ref{xxrem1.2}(2) 
is prime, and is PI if and only if $q$ is a root of unity. 
When $q$ is a primitive $d$-th root of unity, the PI degree 
of $\kk_q[z_1, \dots, z_m]$ is $d^{\lfloor m/2 \rfloor}$ \cite[I.14]{BG}.
\smallskip

\item[(3)] See \cite[p.98]{Row} for the (technical) definition 
of PI degree for~non-prime~PI~rings.
\end{enumerate}
\end{remark}

\begin{lemma}
\label{xxlem1.8}
Recall Notation~\ref{xxnot1.5}. The following statements hold.
\begin{enumerate}
\item[(1)] There is a natural surjective map
$\DD_n\twoheadrightarrow \CC_n$, and there is a natural injective map 
$\CC_n \hookrightarrow  \EE_n^{!}$.
As a consequence, $\EE_n^{!}$ is a module over both $\CC_n$ and $\DD_n$.
\smallskip
\item[(2)]
$\EE_n^{!}$ is a finitely-generated module both over $\CC_n$ and  $\DD_n$.
\smallskip
\item[(3)]
$\DD_n$ satisfies relations
\begin{equation}
\label{E1.8.1}\tag{E1.8.1}
a_{i,j}^2 a_{i,k}=a_{i,j}a_{i,k}^2
\end{equation}
for all distinct $i,j,k$.
\smallskip
\item[(4)]
$\GKdim \DD_n=\lfloor n/2\rfloor$.
\end{enumerate}
\end{lemma}

\begin{proof} 
(1) This follows from Lemma~\ref{xxlem1.4}.

\smallskip

(2) By Lemmas~\ref{xxlem1.3}
and~\ref{xxlem1.4}, we have that as a $\kk$-vector space,
$$\EE_{n}^{!}=\sum_{e_{i,j}=0,1} \kk \; y_{1,2}^{e_{1,2}}y_{1,3}^{e_{1,3}} \cdots
y_{n-1,n}^{e_{n-1,n}} \;  \CC_n.$$ Therefore,
$\EE_{n}^{!}$ is a finitely generated $\CC_n$-module. The last claim 
holds as $\DD_n \twoheadrightarrow \CC_n$.

\smallskip

(3) In $\DD_n$ we have
$a_{i,j} (a_{i,k}-a_{j,k})=0$
for all distinct $i,j,k$. Then we get
$$a_{i,j}a_{i,k}(a_{i,j}-a_{i,k})=a_{i,j}a_{j,k}(a_{i,j}-a_{i,k})=0.$$
as desired.

\smallskip

(4) Let $m:=\lfloor n/2\rfloor$ and 
let $I:=\{(1,2),(3,4), \dots, (2m-1,2m)\}$.
Let $B$ be the quotient algebra $\DD_n/(a_{i,j})_{i<j, (i,j)\not\in I}$.
Observe that $B$ is the commutative polynomial ring 
generated by the set $\{a_{i,j}\}_{(i,j)\in I}$. 
So, $$\GKdim \DD_n\geq \GKdim B=m$$ by Remark~\ref{xxrem1.2}(2,4).

Next, we will show that $\GKdim \DD_n\leq m$. 
Order the generators $\{a_{i,j}\}$ of $\DD_n$ by 
$a_{i,j}<a_{k,l}$ if either $j<l$, or $j=l$ and $i<k$ 
(we assume that $i<j$ and $k<l$ here). Using the 
Diamond Lemma \cite{Be} and the relations in Lemmas~\ref{xxlem1.4} and~\ref{xxlem1.8}(3), we obtain  that if we have a nonzero 
monomial of the form
$$
a_{i_1, j_1}^{r_{i_1, j_1}} \cdots a_{i_{t}, j_{t}}^{r_{i_{t}, j_{t}}}
$$
that is not a linear combination of lower degree terms with 
$r_{i_{s},j_{s}}\geq 2$ for all $s$, then $\{i_{s_1},j_{s_1}\}\cap 
\{i_{s_2},j_{s_2}\}=\emptyset$. 
Therefore, $t\leq m$. Let $V$ be the graded subspace of $\DD_n$ defined by
$$V=\sum \Bbbk a_{i_1, j_1}^{r_{i_1, j_1}} \cdots a_{i_{t}, j_{t}}^{r_{i_{t},j_{t}}}$$
where sum runs over all possible $(i_s, j_s)$ such that 
$\{i_{s_1},j_{s_1}\}\cap \{i_{s_2},j_{s_2}\}=\emptyset$ for any pairs
with $r_{i_{s_1},j_{s_1}}, r_{i_{s_2},j_{s_2}}\neq 0$. By the description of 
monomials above, we get that
$$\DD_{n}=\sum_{e_{i,j} =0,1}\kk \; a_{1,2}^{e_{1,2}}a_{1,3}^{e_{1,3}} 
\cdots 
a_{n-1, n}^{e_{n-1,n}} \; V.$$ Since
$\GKdim V\leq m$, and since $\DD_{n}$ is a quotient
space of a finite direct sum of $V$, we obtain that $\GKdim \DD_{n}\leq 
m$ as desired.
\end{proof}

\bigskip

\begin{proof}[Proof of Theorem \ref{xxthm0.3}] 
(1) By Lemma~\ref{xxlem1.8}(2), $\EE_n^{!}$ is a finitely
generated module over the finitely-generated commutative 
(so, Noetherian) $\kk$-subalgebra  $\mathcal{C}_n$ from 
Notation~\ref{xxnot1.5}.
Thus $\EE_n^{!}$ is a finitely generated Noetherian PI 
algebra, and is module-finite over its center. 
See \cite[Proposition~1.6]{GW}  and Remark~\ref{xxrem1.7}(1).

\smallskip

(2) First, we  construct a factor algebra of  $\EE_n^{!}$ 
to bound its GK-dimension from below. Let $m:=\lfloor n/2\rfloor$ 
and let $I:=\{(1,2),(3,4), \dots, (2m-1,2m)\}$.
Now take $$B:=\EE_{n}^{!}/(y_{i,j})_{i<j, (i,j)\not\in I},$$
a factor algebra of $\EE_{n}^{!}$.
By Definition-Lemma~\ref{xxdeflem0.2} we get that
$B$ is isomorphic to the $(-1)$-skew polynomial ring
$\Bbbk_{-1}[y_{1,2}, y_{3,4}, \dots, y_{2m-1,2m}]$.
So, by Remark~\ref{xxrem1.2}(2,4), 
$$\GKdim \EE_{n}^{!}\geq \GKdim B=m.$$  

On the other hand, since $\EE_{n}^{!}$ is a finite module 
over $\CC_n$, and $\CC_n$ is a homomorphic image of $\DD_n$ 
[Lemma~\ref{xxlem1.8}(1,2)], we obtain that 
$$\GKdim \EE_{n}^{!} = \GKdim \CC_n \leq \GKdim \DD_n=m$$
by Remark~\ref{xxrem1.2}(4) and Lemma~\ref{xxlem1.8}(4).
Thus, $\GKdim \EE_{n}^{!} = m = \lfloor n/2\rfloor.$

\smallskip

(3) This follows from Lemma~\ref{xxlem1.4}: Namely, 
$y_{j,k}(a_{i,j} - a_{i,k}) = 0$ for all 
distinct $i, j, k$, yet  $y_{j,k}$ and $a_{i,j} - a_{i,k}$ 
are nonzero in $\EE_n^!$ (by the Diamond Lemma \cite{Be}).
\end{proof}

\medskip


\section{Results of Etingof on the algebras $\mathcal{C}_n$ and $\mathcal{D}_n$}
\label{xxsec2}

Recall that  $n$ is an integer $\geq 2$. This section contains some results due to Pavel Etingof on the commutative algebras $\mathcal{C}_n$ and $\mathcal{D}_n$ from Notation~\ref{xxnot1.5}.
We thank him for suggesting these results and providing us with the
key ideas for the proofs. 

\smallskip To begin, we need to set some notation on set partitions and corresponding graphs and monomials.

\begin{notation}
[\textnormal{$[n]$}, $\pi$, $\Pi_n$, $B_r$, rk$(\pi)$, $\#_1(\pi)$, $\#_{\geq 2}(\pi)$] 
\label{xxnot2.1}
We fix some notation on set partitions.
\begin{itemize}
\item Let $[n]$ denote the set $\{1,\dots,n\}$. 
\smallskip

\item Let $\Pi_n$ denote the collection of set partitions of $[n]$. Namely, $\pi \in \Pi_n$ if $\pi $ is a set of nonempty subsets $B_1,\dots, B_d$ of $[n]$ so that each positive integer between 1 and $n$ lies in exactly one of the $B_i$.
 We refer to the subsets $B_i$ as the {\it blocks} of $\pi$, and we denote the number of blocks of $\pi$ by $\# \pi$.
\smallskip

\item The {\it rank} of $\pi \in \Pi_n$, denoted by rk$(\pi)$, is the value $n - \#\pi$.

\smallskip

\item We denote by $\#_1(\pi)$ (respectively, $\#_{\geq 2}(\pi)$) the number of blocks of $\pi$ that are singletons (respectively, have cardinality $\geq 2$).
\smallskip

\item We call $\pi \in \Pi_n$ {\it trivial} if $\# \pi = \#_1(\pi) = n$, that is, if $ \pi = \{\{1\}, \{2\}, \dots, \{n\}\}$.

\end{itemize}
\end{notation}

\begin{notation}[$G$, $G_r$, $m_r$, $V(G_r)$, $S(G)$, $T(G)$]  \label{xxnot2.2} Let $G$  be a graph with the vertex set $[n]$ subject to the following conditions. It is loopless and we allow for multiple edges between two vertices. Moreover, corresponding to $\pi \in \Pi_n$ as in Notation~\ref{xxnot2.1}:

\begin{itemize}
\item The graph $G$ has $\#\pi$ connected components of $G$ denoted by $\{G_r\}_{r=1}^{\#\pi}$.

\smallskip

\item Let $m_r = m_r(G_r)$ be the total number of edges 
in each connected component $G_r$ for each $r$. 
\smallskip

\item We refer to $\pi$ as the {\it support} 
of $G$, sometimes denoted by $S(G)$, in the sense that the block $B_r$ is the 
vertex set $V(G_r)$ of $G_r$. 
\smallskip

\item The {\it type} of $G$, denoted by $T(G)$, 
is the collection $\{(B_r:=V(G_r),m_r)\}_{r=1}^{\# \pi}$. 
\end{itemize}
\end{notation}

\begin{notation}[$G(\underline{f})$, $G_r(\underline{f})$, $m_r(\underline{f})$, $V(G_r(\underline{f}))$, $S(\underline{f})$, $T(\underline{f})$]  \label{xxnot2.3}
Let $F$ be a not necessarily commutative algebra generated by $\{f_{i,j}\mid 1\leq i<j \leq n\}$.
\begin{itemize} 
\item For each monomial $\underline{f}:=f_{i_1,j_1}f_{i_2,j_2}\cdots f_{i_w,j_w}$, we define a graph, denoted $G(\underline{f})$, with vertex set $[n]$, and with edges $i_s$ \textemdash ~$j_s$ for each $f_{i_s,j_s}$ in $\underline{f}$.
\end{itemize}
Then $G(\underline{f})$ is a disjoint union of connected components $\{G_r(\underline{f})\}_{r=1}^d$, 
for some $d \in \mathbb{N}$.
\begin{itemize} 
\item Let $m_r(\underline{f})$ be the total number of edges 
in $G_r(\underline{f})$ for each $r$. 
\item The {\it type} of $\underline{f}$ is defined to be the collection of pairs
$$T(\underline{f}):=T(G(\underline{f}))=\left \{\left(V(G_r(\underline{f})),m_r(\underline{f})\right)\right \}_{r=1}^d$$
and 
the {\it support} of $\underline{f}$ is the set partition of the vertex set of $G(\underline{f})$,
$$S(\underline{f}):=S(G(\underline{f}))=\left \{V(G_r(\underline{f}))\right \}_{r=1}^d.$$
\end{itemize}
\end{notation}

\medskip

We define a partial ordering on the types of graphs defined above. 

\begin{definition} \label{xxdef2.4}
Choose two graphs $G$ and $H$ on $n$ vertices as in Notation~\ref{xxnot2.2}; the same will apply for the graphs defined on monomials of the same degree as in Notation~\ref{xxnot2.3}.
We write $T(G)<T(H)$ if either:
\begin{enumerate}
\item[(a)]
The set partition $\pi_G \in \Pi_n$ corresponding to $T(G)$ is a {\it proper refinement} of $\pi_H \in \Pi_n$ of $T(H)$ (i.e., each block of $\pi_G$ is a subset of a block of $\pi_H$, one such being a proper subset), or
\smallskip

\item[(b)]
We have $\pi_G = \pi_H$ (i.e.,
$G$ and $H$ have the same support) and the sequence
$\{m_r(G_r)\}_{r=1}^d$ is less than $\{m_r(H_r)\}_{r=1}^d$ in the
lexicographic order.
\end{enumerate}
We also write $S(G)<S(H)$ if condition (a) holds.
\end{definition}

Next we turn our attention to the commutative algebra ${\mathcal D}_n$ from Notation~\ref{xxnot1.5} (which is studied in Lemma~\ref{xxlem1.8}). Recall that ${\mathcal D}_n$ is generated by commuting elements
$\{a_{i,j} \mid 1\leq i
< j\leq n\}$ and is subject to the relations
\begin{align}
\label{E2.4.1}\tag{E2.4.1}
a_{i,j}a_{j,k}&=a_{i,j}a_{i,k} \; \quad \; \forall \; 
{\text{distinct}} \; i,j,k.
\end{align}

We will use Bergman's Diamond lemma \cite{Be} in the next proof. Let 
$\mathcal{S}$ is a reduction system for the set of relations $\mathcal{R}$ of a given finitely presented algebra.
If all ambiguities of $\mathcal{R}$ can be resolved 
 using $\mathcal{S}$,  then we call $\mathcal{S}$ a {\it Gr{\"o}bner 
basis} of $\mathcal{R}$.

\begin{lemma}
\label{xxlem2.5}
Recall the commutative algebra ${\mathcal D}_n$ from Notation~\ref{xxnot1.5}.
Every monomial of ${\mathcal D}_n$ (in the variables $a_{i,j}$) is nonzero.
\end{lemma}

\begin{proof} This basically follows from Bergman's Diamond
lemma \cite{Be}. 
Rewrite the set of relations of ${\mathcal D}_n$  as 
\begin{equation}
\notag
a_{i_1,j_1}a_{i_2,j_2}=a_{k_1,l_1} a_{k_2,l_2}
\end{equation}
for some indices $(i_1,j_1)$, $(i_2,j_2)$, $(k_1,l_1)$ and
$(k_2,l_2)$. By induction, every relation in the Gr{\"o}bner 
basis derived from overlap ambiguity is of the form
\begin{equation}
\label{E2.5.1}\tag{E2.5.1}
a_{i_1,j_1}\cdots a_{i_w,j_w}=a_{k_1,l_1} \cdots a_{k_w,l_w}
\end{equation}
for some pairs of indices $(i_s,j_s)$, $(k_s,l_s)$, for
$s=1,\dots, w$. It is not necessary to specify what are 
these indices. In other words, the reduction in the Diamond 
Lemma only uses the relations of the form \eqref{E2.5.1}. 
This implies that every monomial is equal to a reduced 
monomial. Therefore every monomial of ${\mathcal D}_n$ is nonzero.
\end{proof}

We now discuss the growth of the algebra $\mathcal{D}_n$.

\begin{proposition} \label{xxpro2.6}
The Hilbert series of ${\mathcal D}_n$  is
$$H_{{\mathcal D}_n}(t)~~=~~\sum_{\pi \in \Pi_n} \frac{t^{ \textnormal{rk}(\pi)}}{(1-t)^{\#_{\geq 2}(\pi)}} ~~=~~
1+ \sum_{ \textnormal{non-triv~}\pi \in \Pi_n} \frac{t^{ \textnormal{rk}(\pi)}}{(1-t)^{\#_{\geq 2}(\pi)}}.$$
\end{proposition}

\begin{proof}
Take $\underline{a}:=\prod_{s=1}^w a_{i_s,j_s}$, a monomial of ${\mathcal D}_n$; it is fine to use $\prod$ since $\DD_n$ is commutative. For 
$B=\{i_1,\dots, i_q\}\subseteq [n]$ and $m\geq q-1$,
let  $$\underline{a}(B,m):= a_{i_1,i_2}^{m-q+2} a_{i_1,i_3}
\cdots a_{i_1,i_q}$$
be a certain monomial of $\DD_n$ of degree $m$. By convention, $\underline{a}(B,m)=1$ if $B$ is a singleton
and consequently, $m=0$. By using the relations \eqref{E2.4.1}, 
one sees that $\underline{a}$ is equal to $\underline{a}(B,m)$ in
${\mathcal D}_n$ if and only if $G(\underline{a})$ is connected except for
singletons.

\smallskip Further, let $\underline{a}$ be a monomial with 
$T(\underline{a})= \{(B_r:=V(G_r), m_r)\}_{r=1}^d$, then we claim that
\begin{equation}
\label{E2.6.1}\tag{E2.6.1}
\underline{a}=\prod_{m_r\geq 1} \underline{a}(B_r,m_r) =\prod_{r} \underline{a}(B_r, m_r)
\end{equation}
in the algebra ${\mathcal D}_n$. In other words, a monomial
is determined by its type uniquely. In equation \eqref{E2.6.1}, 
$B_r$s are disjoint subsets of $[n]$. 

\smallskip To show the above claim,  we first show that every element 
of the form 
\begin{equation}
\label{E2.6.2}\tag{E2.6.2}
\underline{a}(B_1,m_1) \cdots \underline{a}(B_d, m_d)
\end{equation} 
is reduced. Every step in the reduction of the Diamond Lemma 
\cite{Be} uses one of commuting relations of ${\mathcal D}_n$ or one of the 
relations in \eqref{E2.4.1} such that $(i,j,k)$ are one of 
the connected components of $G(\underline{a})$. This means that in every step
$\underline{f}=\underline{f}'$ in the reduction process, we have $T(\underline{f})=T(\underline{f}')$. Therefore $T(\underline{a})=T(\underline{a}')$
if $\underline{a}=\underline{a}'$ for two monomials in ${\mathcal D}_n$. 
Clearly, $\underline{a}(B_1,m_1) \cdots \underline{a}(B_d, m_d)$ is the smallest 
with respect to the lexicographic order. Therefore 
$\underline{a}(B_1,m_1) \cdots \underline{a}(B_d, m_d)$ is reduced. 
Secondly, if $\underline{a}$ is a monomial that is not in the form of 
\eqref{E2.6.2}, it can be reduced to another monomial of 
the form of \eqref{E2.6.2} by induction on the degree of 
the monomial. Thus we proved the claim. 

\smallskip
Now the assertion 
follows by counting monomials of the form in \eqref{E2.6.1}. Namely, fix $\pi \in \Pi_n$. Then the  monomials of positive degree  of the form in \eqref{E2.6.1} corresponding to $\pi = \{B_1, \dots, B_{\# \pi}\}$ are $\prod_{r \textnormal{ with } |B_r| \geq 2~ } \underline{a}(B_r, m_r)$; this contributes to the Hilbert series of ${\mathcal D}_n$ the term
$$\prod_{r \textnormal{ with } |B_r| \geq 2~ } (t^{|B_r|-1} + t^{|B_r|} + t^{|B_r| + 1} + \dots) = 
\prod_{r \textnormal{ with } |B_r| \geq 2~ } \frac{t^{|B_r|-1}}{1-t}$$
since the degree of $\underline{a}(B_r, m_r)$ is $m_r$ which has lowest value $|B_r|-1$. 
Since 
\[
\begin{array}{rl}
\sum_{r \textnormal{ with } |B_r| \geq 2~ } (|B_r| - 1) &= \left(\sum_{r \textnormal{ with } |B_r| \geq 2~ } |B_r|\right) - \#_{\geq 2}(\pi)\\
 & = \left(n - \#_1(\pi)\right) - \#_{\geq 2}(\pi)\\
 & = \textnormal{rk}(\pi),
\end{array}
\]
we have that the monomials of positive degree corresponding to a partition $\pi \in \Pi_n$ contributes 
$$ \prod_{r \textnormal{ with } |B_r| \geq 2~ } \frac{t^{|B_r|-1}}{1-t} ~=~ \frac{t^{\textnormal{rk}(\pi)}}{(1-t)^{\#_{\geq 2}(\pi)}}$$
to the Hilbert series of $\mathcal{D}_n$.
\end{proof}

Naturally, we then inquire:

\begin{question} \label{xxque2.7}
What is the Hilbert series of ${\mathcal E}_n^!$?
\end{question}

\begin{proposition}[Etingof]
\label{xxpro2.8} 
The algebra ${\mathcal D}_n$ is reduced, that is,  
${\mathcal D}_n$  has no nonzero nilpotent elements. 
As a consequence, ${\mathcal D}_n$ is semiprime.
\end{proposition}

\begin{proof} 
Let $0\neq f\in {\mathcal D}_n$ be a linear combination of 
monomials $\sum_{i} c_i \underline{a}_i$, for $c_i \in \kk$. 
Here $\underline{a}_i\in {\mathcal D}_n$ are linearly independent 
monomials in $a_{i,j}$s.
We want to show that 
$f^m\neq 0$ for all $m\geq 1$. We pick a summand $c_{i_0}\underline{a}_{i_0}\neq 0$ 
of $f$ such that $T(\underline{a}_{i_0})$ is one of minimal type according to the partial ordering of Definition~\ref{xxdef2.4}. We want to 
show that $f^m$ is a linear combination of $c_{i_1}\cdots c_{i_m}
\underline{a}_{i_1}\cdots \underline{a}_{i_m}$ with its summand $c_{i_0}^m \underline{a}_{i_0}^m \neq 0$ 
being of minimal type. This is enough to conclude
that $f^m\neq 0$ by Lemma \ref{xxlem2.5}.

\smallskip Consider the term $\underline{a}_{i_1}\cdots \underline{a}_{i_m}$ that is not
$\underline{a}_{i_0}^{m}$ and suppose by way of contradiction that $T(\underline{a}_{i_1}\cdots \underline{a}_{i_m}) < 
T(\underline{a}_{i_0}^m)$. Then $S(\underline{a}_{i_1}\cdots \underline{a}_{i_m}) < 
S(\underline{a}_{i_0}^m)=S(\underline{a}_{i_0})$ by condition~(a) of Definition~\ref{xxdef2.4}. This implies that 
$S(\underline{a}_{i_s})\leq S(\underline{a}_{i_0})$ for all $s$. Since $T(\underline{a}_{i_0})$ 
is minimal, we obtain that $S(\underline{a}_{i_s})=S(\underline{a}_{i_0})$. When all 
$\underline{a}_{i_s}$ have the same support, we use condition~(b) of Definition~\ref{xxdef2.4}
to obtain that $T(\underline{a}_{i_s}) < T(\underline{a}_{i_0})$ for some 
$\underline{a}_{i_s}\neq \underline{a}_{i_0}$, a contradiction. Thus $f^m$ is a linear 
combination of monomials with one summand $c_{i_0}^m \underline{a}_{i_0}^{m}$ of 
summand with minimal type. The assertion follows.

\smallskip
The semiprimeness of $\mathcal{D}_n$ now holds since $\mathcal{D}_n$ is commutative.
\end{proof}

We need the following construction for the next 
result. 

\begin{definitionlemma}
\label{xxdeflem2.9} 
Let $\{R_i\}_{i=1}^{d}$ be some ${\mathbb Z}_2$-graded algebras. The {\it super tensor product} between $R_i$ and $R_j$ is a ${\mathbb Z}_2$-graded algebra $$R_i \otimes_{\textnormal{super}} R_j,$$ which is equal to  $R_i \otimes R_j$ as ${\mathbb Z}_2$-graded algebras, where 
for homogeneous elements $f,f'\in R_i$ and $g,g'\in R_j$
we have
$$(f\otimes g)(f'\otimes g')=(-1)^{|g||f'|} (ff'\otimes gg').$$
 The ${\mathbb Z}_2$-grading of $R_i \otimes_{\textnormal{super}} R_j$ is $(R_i \otimes_{\textnormal{super}}  R_j)_0 =
(R_i)_0\otimes (R_j)_0 + (R_i)_1 \otimes (R_j)_1$
and 
$(R_i \otimes_{\textnormal{super}}  R_j)_1  = 
(R_i)_1\otimes (R_j)_0 + (R_i)_0 \otimes (R_j)_1.$
Moreover, $R_i\otimes_{\textnormal{super}} R_j$ is a graded twist of $R_i\otimes R_j$ in the sense 
of \cite{Zh}.
\qed

\end{definitionlemma}

Recall that $\CC_n$ is the central subalgebra of $\EE_n^{!}$
generated by $a_{i,j}:=y^2_{i,j}$ for all $i,j$.

\begin{proposition}[Etingof]
\label{xxpro2.10}
There is a natural algebra isomorphism
${\mathcal D}_n\cong {\mathcal C}_n$.
\end{proposition}

\begin{proof} 
We achieve the result by studying the points of $$X_n:=\Spec {\mathcal D}_n
\subseteq \Bbbk^{{n \choose 2}}.$$
Let us set some notation used in this proof. Given a point $z =(z_{i,j})_{1\leq i<j \leq n}\in X_n$, let $G(z)$ be the graph on 
the vertex set $[n]$ where $i$ and $j$ are connected 
if and only if  $z_{i,j}\ne 0$. Let $\{G_r(z)\}_{r=1}^d$ be
the set of connected components of $G(z)$. The relations 
$$z_{i,j}(z_{i,k}-z_{j,k})=0$$
(derived from \eqref{E2.4.1}) imply that, for  some $x_r\in \Bbbk^{\times}$,
\begin{equation} 
\label{E2.10.1}\tag{E2.10.1}
z_{i,j} = \begin{cases}
 \text{$x_r$}, &\text{ if 
$i,j$ belong to the same component $G_r(z)$},\\ 
0, &\text{ otherwise}
\end{cases}
\end{equation}
 by definition. Thus each connected component
$G_r(z)$ is a complete graph. 

\smallskip  On the other hand, for each set
partition $\pi = \{ B_1, \dots, B_d \} \in \Pi_n$ we 
can construct a stratum 
$$X_n(\pi):= \{z \in  X_n ~|~ G_r(z) \text{ is the complete graph } K(B_r)  ~\forall r\}.$$
It is clear that $X_n(\pi)$ is a torus of dimension equal to 
$\#_2(\pi)$.

\smallskip
For the rest of the proof, recall that by Lemma~\ref{xxlem1.8}(1) there is a natural algebra surjection from
$\phi: {\mathcal D}_n\to {\mathcal C}_n$. So, the assertion is equivalent
to the statement that the map $\phi$ is injective. Let $I$ be the kernel of $\phi$
and $\mathcal{Z}(I) \subseteq X_n$ be the zero set of $I$. Since ${\mathcal D}_n$ is reduced [Proposition~\ref{xxpro2.8}(1)], it suffices to show that $\mathcal{Z}(I) =X_n$. 

\smallskip Note that $I$ is also equal to the kernel of the composition 
$({\mathcal C}_n \hookrightarrow {\mathcal E}_n^!) \circ \phi$, 
where the inclusion holds by Lemma~\ref{xxlem1.8}(1). 
Take a point $z \in X_n$ and consider the corresponding 
maximal ideal $\mathfrak{m}_z$ of ${\mathcal D}_n$ and 
form the algebra
$${\mathcal E}_n^!(z) := {\mathcal E}_n^! \otimes_{{\mathcal D}_n} 
{\mathcal D}_n/(\mathfrak{m}_z).$$
Then it suffices to show that ${\mathcal E}_n^!(z) \neq 0$ 
for all $z \in X_n$. Namely, this condition is equivalent 
to the condition that $\mathfrak{m}_z$ contains $I$ for 
all $z \in X_n$, which in turn is equivalent to 
$I \subset \bigcap_{z \in X_n} \mathfrak{m}_z$; here, 
$ \bigcap_{z \in X_n} \mathfrak{m}_z = 0$ since 
${\mathcal D}_n$ is reduced.

\smallskip Let $z$ be any point in $X_n$. Then it  
belongs to a stratum $X_n(\pi)$ for some 
set partition $\pi = \{B_1, \dots, B_{\# \pi}\}$ 
of $[n]$, where $B_r = G_r(z)$. Then, 
${\mathcal E}_n^!(z)$ is defined as the quotient 
algebra 
$${\mathcal E}_n^!(z) 
= {\mathcal E}_n^!/\left(\{y_{i,j}^2 - x_r\}_{i,j  \in G_r(z)},
\; \{y_{i,j}^2\}_{i \text{ or } j ~ \not \in G_r(z)}\right), 
\quad \text{for $x_r$ in  \eqref{E2.10.1}}.$$
Let us replace the last relations by an even stronger relation: 
Take the ideal 
$$J = \left(y_{i,j}\right)_{\text{$i,j$ in different blocks of $\pi$}}.$$
Then the resulting quotient algebra ${\mathcal E}_n^!(z)/J$ 
is a super tensor product of algebras, denoted by $R_r$, 
associated to 
components $B_r$ with $|B_r| \geq 2$. Now by 
Definition-Lemma \ref{xxdeflem2.9}, 
it suffices to show that each $R_r$ is nonzero, or equivalently, 
to show that ${\mathcal E}_n^!(z)$ is nonzero when $\# \pi = 1$.

\smallskip Without loss of generality, we can first assume that 
$\# \pi = 1$ and, by 
replacing $y_{i,j}$ by $x_1^{-1/2}y_{i,j}$, we may further 
assume that $x_1=1$. Now  $z_{i,j}=1$ for all $i\neq j$ 
by~\eqref{E2.10.1}. So it suffices to show that the algebra
${\mathcal E}_n^!(1)$ is nonzero.

\smallskip For $n=3$, one can check directly or by using a 
computer algebra program that the algebra
$${\mathcal E}_3^!(1) = 
\frac{\kk \langle y_{1,2}, \; y_{1,3}, \; y_{2,3} \rangle}{ \left( 
\begin{array}{c} 
y_{1,2}^2 - 1, \quad y_{1,3}^2 - 1, \quad y_{2,3}^2- 1, \\
y_{1,2}y_{2,3} + y_{2,3}y_{1,3}, \quad
y_{1,2}y_{2,3} + y_{1,3}y_{1,2}, \\
y_{2,3}y_{1,2} + y_{1,3}y_{2,3}, \quad
y_{2,3}y_{1,2} + y_{1,2}y_{1,3}
\end{array}\right)}$$
is nonzero. For $n \geq 4$, we can check that the algebra 
${\mathcal E}_n^!(1)$ is isomorphic to a factor of the group 
algebra of a Schur cover $2\cdot S_n^+$ of the symmetric 
group $S_n$. Namely, the group 
$2\cdot S_n^+$ has a presentation with generators 
$w, s_1, \dots, s_{n-1}$ and relations
\[
\begin{array}{ll}
w^2 = 1, &\\
\notag
w s_i = s_i w \quad \text{ and } \quad s_i ^2 = 1 \; &\forall \; 1 \leq  i \leq  n-1,\\
\notag
s_{i+1}s_{i}s_{i+1} = s_{i}s_{i+1}s_{i}, \; &\forall\; 1 \leq i \leq n-2,\\
\notag
s_{j}s_{i} = s_{i}s_{j}w, \; &\forall \; 1 \leq i < j \leq n-1, ~|i-j| \geq 2;
\end{array}
\]
see \cite[Chapter~12]{Karp}. Now by identifying $w = -1$ 
and $s_i=y_{i,i+1}$ we get $${\mathcal E}_n^!(1)\cong \Bbbk (2\cdot S_n^+)/(w+1).$$
Since $|2\cdot S_n^+| = 2n!$ and the order of $w$ is 2, we have that
$\Bbbk (2\cdot S_n^+)/(w+1)$ has dimension $n!$, so ${\mathcal E}_n^!(1)$ is not zero, as desired.
\end{proof}

One quick consequence of the proposition above is the following.

\begin{corollary}
\label{xxcor2.11}
Every monomial in $\EE_{n}^{!}$ is nonzero.
\end{corollary}

\begin{proof}
Since the natural map $\DD_n\to \EE_n^{!}$
is injective [Lemma~\ref{xxlem1.8}(1) and Proposition~\ref{xxpro2.10}], every element of the form $y_{i_1,j_1}^2 y_{i_2,j_2}^2\cdots 
y_{i_w,j_w}^2$ is nonzero by Lemma \ref{xxlem2.5}. 
Starting from any monomial  $y_{i_1,j_1}y_{i_2,j_2}\cdots
y_{i_w,j_w}$ in $\EE_n^{!}$, one can see that 
$$(y_{i_1,j_1}y_{i_2,j_2}\cdots y_{i_w,j_w})
(y_{i_w,j_w}\cdots y_{i_2,j_2}y_{i_1,j_1})
=y_{i_1,j_1}^2 y_{i_2,j_2}^2\cdots 
y_{i_w,j_w}^2\neq 0$$
in $\EE_n^{!}$. Therefore 
$y_{i_1,j_1}y_{i_2,j_2}\cdots y_{i_w,j_w}\neq 0$.
\end{proof}

Finally, we recover a bound on the GK-dimension of $\mathcal{E}_n^!$ via the proof of Proposition~\ref{xxpro2.10}.

\begin{remark}[Etingof] \label{xxrem2.12} 
We refer to the notation of the proof of Proposition~\ref{xxpro2.10}.
By taking limits, one sees that one stratum $X_n(\pi')$ 
is in the closure of another stratum $X_n(\pi)$ if and only if  
$\pi'$ is obtained from $\pi$ by cutting some blocks of $\pi$ 
into single points. Thus, the irreducible components of 
$X_n$ are closures $\overline{X_n(\pi)}$ of strata with 
maximal set partitions $\pi$, i.e., those with at most 
one $1$-point block. So, the dimension of the stratum 
$X_n(\pi)$ is $\#_{\geq 2}(\pi)$, which is at most $\lfloor n/2\rfloor$. We then recover that $\text{GKdim}{\mathcal E}_n^! = \text{GKdim}{\mathcal D}_n \leq \lfloor n/2\rfloor$; see Lemma~\ref{xxlem1.8}(4) and Theorem~\ref{xxthm0.3}(2).
\end{remark}

\bigskip

\section{Homological preliminaries and Proof of 
Theorem~\ref{xxthm0.6}\\ and Corollary~\ref{xxcor0.8}}
\label{xxsec3}

The goal of this section is to establish 
Theorem~\ref{xxthm0.6} and Corollary~\ref{xxcor0.8} 
on various homological properties of $\EE_n^!$. We 
begin by recalling these homological notions for 
general connected $\mathbb{N}$-graded algebras, and 
show how these conditions are related in Figure~\ref{fig:hom}. 
Then, we  present preliminary results for $\EE_n^!$, 
and we end the section by establishing the proofs of 
Theorem~\ref{xxthm0.6} and Corollary~\ref{xxcor0.8}.

\smallskip

We begin by presenting homological conditions on connected 
$\mathbb{N}$-graded (c.g.), locally-finite $\kk$-algebras  
that generalize the regularity, Gorenstein, Cohen-Macaulay, 
and other favorable conditions on commutative local algebras. 
Some of these homological conditions can be defined for 
$\kk$-algebras that are not necessarily connected 
$\mathbb{N}$-graded nor locally-finite, and we refer 
the reader to the references provided below for more information.

\begin{definition} 
\label{xxdef3.1} 
\cite[Intro.]{AS} \cite[page~674, Section~4]{vdB} 
\cite[Intro., Def.~2.1 and~5.8]{Lev}  
Let $A$ be a connected $\mathbb{N}$-graded (c.g.), 
locally-finite $\kk$-algebra. All $A$-modules and 
Ext-groups below will be graded.
\smallskip

\begin{enumerate}
\item[(1)] $A$ is called {\it Artin-Schelter-Gorenstein} 
(or {\it AS-Gorenstein}) (of dimension $d$) if 
the following conditions hold:
\begin{enumerate}
\item[(a)]
$A$ has finite injective dimension $d$ as both a left and right $A$-module,
\item[(b)]
$\Ext^i_A(\Bbbk, A) = \Ext^i_{A^{op}} (\Bbbk, A) = 0$ for all 
$i\neq d$, and
\item[(c)] 
$\Ext^d_A(\Bbbk, A)\cong \Bbbk(l)$
and $\Ext^d_{A^{op}} (\Bbbk, A)\cong
\Bbbk(l)$ for some integer $l$.
\end{enumerate}

\medskip 

\item[(2)]  $A$ is called {\it Artin-Schelter-regular} 
(or {\it AS-regular}) of dimension $d$ if $A$ is 
AS-Gorenstein of finite global dimension $d$.
(We do not assume finite 
GK-dimension as was  introduced originally in \cite{AS}. 
See \cite{JZ}, for instance.)

\medskip 

\item[(3)] For $\mathfrak{m}:= A_{\geq 1}$, a graded 
maximal ideal of $A$, the {\it $i$-th local cohomology module} 
for a left (or right) 
$A$-module $M$ is defined to be
$$H_{\mathfrak{m}}^i(M) := \lim_{\overset{\longrightarrow}{\text{\tiny $m$}}} 
\Ext^i_A(A/A_{\geq m}, ~M).$$

\smallskip 

\item[(4)] $A$ is {\it Artin-Schelter-Cohen-Macaulay} 
(or {\it AS-Cohen-Macaulay}, {\it AS-CM}) if there exists 
an integer $d$ such that $H_{\mathfrak{m}}^i(A) = H_{\mathfrak{m}^{op}}^i(A) = 0$ 
for all $i \neq d$.

\medskip 

\item[(5)] The {\it grade number} of a left (or right) 
$A$-module $M$ is defined to be
$$j_A(M) := \inf\{i~|~\Ext_A^i(M,A)\neq0\}\in {\mathbb N}\cup\{+\infty\}.$$ 
Write $j(M)$ for $j_A(M)$ if $A$ is understood. Note that 
$j_A(0)=+\infty$.

\medskip 

\item[(6)] In the case when $A$ is Noetherian, a left (right) $A$-module $M$ 
satisfies the {\it Auslander condition} if for any $q\geq0$ we get that
$j_A(N)\geq q$ for all right (left) 
$A$-submodules $N$ of $\Ext_A^q(M,A)$. 

\medskip

\item[(7)] A Noetherian algebra $A$ is {\it Auslander-Gorenstein}  of 
dimension $d$ if $A$ has finite injective dimension $d$ as both a left 
and right $A$-module, and if every 
finitely-generated left and right $A$-module satisfies the Auslander condition.

\medskip 

\item[(8)] A Noetherian algebra $A$ is {\it Auslander-regular}  
of dimension $d$ if $A$ is Auslander-Gorenstein of finite global dimension $d$.

\medskip 

\item[(9)] A Noetherian algebra $A$ is {\it Cohen-Macaulay} ({\it CM})  
if $\GKdim(A)=d\in{\mathbb N},$ and if $$j(M)+\GKdim(M)=d$$ 
for every finitely-generated nonzero left (or right) $A$-module $M$. 

\medskip 

\item[(10)] Continuing (9), we have {\it inequality of the grade} ({\it IG}) if the weaker condition that $j(M)+\GKdim(M) \geq d$ is satisfied.

\medskip 

\item[(11)] The {\it depth} of a left (or right) $A$-module $M$ is defined to be
$$\depth M:=\inf \{i\mid \Ext^i_A(\Bbbk, M)\neq 0\}.$$
If $\Ext^i_{A}(\Bbbk, M)=0$ for all $i$, then $\depth M=\infty$.

\end{enumerate}
\end{definition}

The hierarchy of homological conditions on connected $\mathbb{N}$-graded, 
locally-finite $\kk$-algebra also involve certain factors of regular 
algebras, namely, the complete intersections discussed below. This 
motivated by the fact that in the context of commutative local rings,  
for  a Noetherian ring $R$ we have that:
\begin{itemize}
\item If $R$ is regular, then $R$ is a complete intersection; and 
\item If $R$ is a complete intersection, then $R$ is Gorenstein, 
and hence in turn, is Cohen-Macaulay.
\end{itemize} 
See, for instance, work of Bass \cite{Bas} for the details of the 
commutative terminology; Definitions~\ref{xxdef3.1} and~\ref{xxdef3.2} 
are a noncommutative generalization of these concepts. 

\begin{definition}  \label{xxdef3.2} \cite{GW} \cite[Definition 1.3]{KKZ} 
\begin{enumerate}
\item[(1)] An element $\Omega$ of a ring $A$ is called {\it normal} 
if $\Omega A=A \Omega$, and is called {\it regular} if $\Omega$ is 
a nonzero divisor in $A$.

\medskip

\item[(2)] We say a collection of elements 
$\{\Omega_1,\dots, \Omega_t\}$ of a ring $A$ is a {\it normal
sequence} in $A$ if $\deg \Omega_i>0$ for all $i$, and if
$\Omega_i$ is a normal element in the
factor ring $A/(\Omega_1,\dots, \Omega_{i-1})$
for all $i$. If, further, each $\Omega_i$ is a regular 
element in the factor ring $A/(\Omega_1,\dots, \Omega_{i-1})$, then
we say that $\{\Omega_1,\dots, \Omega_t\}$ is a {\it regular normal
sequence} in $A$.

\medskip

\item[(3)] A c.g. finitely-generated $\kk$-algebra $A$ is called a 
{\it classical complete 
intersection} ({\it cci}) if there is a c.g. Noetherian AS-regular 
algebra $C$ and a regular 
sequence of normalizing elements $\{\Omega_1,\dots, \Omega_t\}$
of positive degree such that
$$A\cong C/(\Omega_1,\dots, \Omega_t).$$

\smallskip

\item[(4)] In the case when $A$ is a cci, the {\it cci number} of 
$A$ is defined to be
$$cci(A) = \min\{t \mid A \cong C/(\Omega_1,\dots, \Omega_t)\}.$$
\end{enumerate}
\end{definition}

\smallskip

Now we illustrate the connections amongst the terminology 
above  in the case when $A$ is a connected $\mathbb{N}$-graded, 
locally-finite $\kk$-algebra. One fact that is used in the 
hypotheses of some of the references below is that a Noetherian, 
PI algebra is {\it fully bounded Noetherian (FBN)} 
\cite[Definition~6.4.7, Corollary~13.6.6]{MR}.

\smallskip

\bigskip

{\small \begin{figure}[h!] 
\[
\xymatrix{
\text{\framebox{Ausl.-reg.}} 
\ar@/^{1.5pc}/@{=>}[rr]^{\hspace{.1in} \substack{\hspace{-.13in}\text{by def: } \text{gldim$< \hspace{-.05in}\infty \Rightarrow$ injdim$< \hspace{-.05in}\infty$}}}  
\ar@{=>}[dd]
^{\substack{\textcolor{white}{.}\\\textcolor{white}{.}\\\text{\cite[Thm~6.3]{Lev}}}
}
 && \text{\framebox{Ausl.-Goren.}} \ar@{=>}[dd]
 ^{\substack{\text{\cite[Thm~6.3]{Lev}}}
 }
 &  \text{\framebox{IG}}\ar@{<=>}[rrd] \ar@{<=>}[dd]^{\hspace{.13in}\substack{\text{\cite[Sec~5]{Jo1}} \\ \text{\cite[T2.2]{Jo2}}   \\ 
 A \text{ Noeth, PI}}}&& \text{\framebox{CM}} \ar@{=>}[ll]_{\text{by def}} \\
 &&&&&\text{\framebox{{\Large$\substack{\depth A\\ =\\ \GKdim A}$}}} \ar@{<=>}[lld]\\
\text{\framebox{AS-reg.}} \ar@/_{1.5pc}/@{=>}[rr]_{\hspace{-.25in} \substack{\hspace{.3in}\text{by def:}\text{ gldim$< \hspace{-.05in}\infty \Rightarrow$ injdim$< \hspace{-.05in}\infty$}}}
\ar@{=>}[r]^{\hspace{.1in}\substack{\text{by def:}\\ (C = A,\\ \{\underline{\Omega}\} = \emptyset)}} 
&\text{\framebox{cci}} \ar@/^1.5pc/@{=>}[ruu]^{\substack{\text{\cite[Thm~3.6]{Lev}\hspace{-.2in}}\\ C \text{ Ausl.-Gor. \hspace{-.25in}} 
\\ \textcolor{white}{.}\\ \textcolor{white}{.}\\ \textcolor{white}{.}
}} \ar@/^{.7pc}/@{=>}[r]^{\substack{\text{\cite[Thm~3.4]{KKZ}}\\A = R^G
\text{ inv. ring}}} &  \text{\framebox{AS-Goren.}} \ar@/_{1pc}/@{=>}[r]_{\text{\cite[Lem~2.1]{JZ}}}&  \text{\framebox{AS-CM}}&&
}
\]
\caption{Homological conditions for c.g. locally-finite algebras $A$ [Def.~\ref{xxdef3.1}]} \label{fig:hom}
\end{figure}
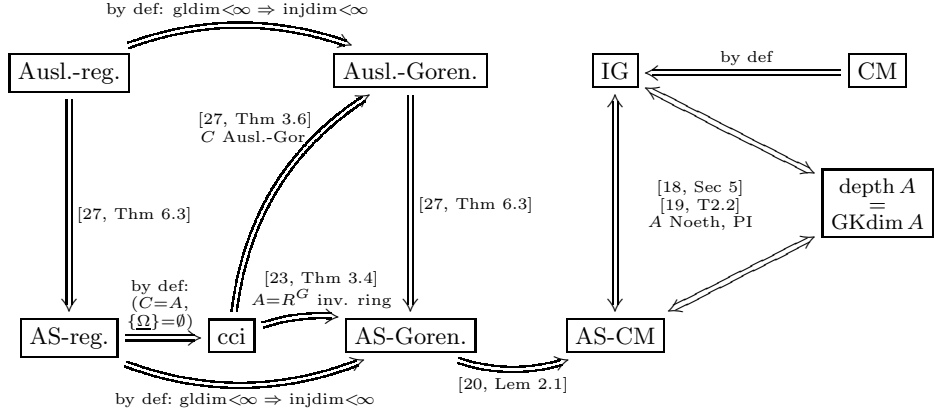 }

\smallskip

\bigskip

The following result will be of use.

\begin{proposition} 
\label{xxpro3.3} \cite[Lemma~5.7]{Lev} \cite[Theorem~2.5]{CV} 
Take $C$ a connected $\mathbb{N}$-graded Noetherian algebra 
with finite GK-dimension.
\begin{enumerate}
\item[(1)] Let $\Omega$ be a regular element of $C$ of 
positive degree, then $\GKdim(C/(\Omega)) = \GKdim(C) - 1$.

\smallskip

\item[(2)] Let $m=\GKdim (C)$.
Suppose that $C$ is Auslander-Gorenstein 
and Cohen-Macaulay, and that there is a normalizing sequence 
$\{\Omega_1, \dots, \Omega_{m}\}$ of $C$ with each 
element $\Omega_i$ homogeneous of positive degree. 
Then, $\{\Omega_1, \dots, \Omega_{m}\}$ is a regular 
sequence in C if and only if  $C/(\Omega_1, \dots, \Omega_{m})$  
is finite-dimensional (has GK-dimension 0). 
In this case, for each $t = 1, \dots, m$, we get that  
$\{\Omega_1, \dots, \Omega_t\}$ is a regular sequence of $C$ 
if and only if $$\GKdim(C/(\Omega_1, \dots, \Omega_t)) = m - t.$$ 

\vspace{-.25in}
\qed
\end{enumerate}
\end{proposition}

Now we study the Fomin-Kirillov algebra $\EE_3$ from 
Definition~\ref{xxdef0.1} and show that it is a classical 
complete intersection.
Recall that $\EE_3$ is the $\kk$-algebra generated by 
$x_{12}, x_{13}, x_{23}$ (where we suppress the comma 
in the subscript of indices) subject to the following relations:
$$\begin{aligned}
x_{12}^2=x_{13}^2=x_{23}^2&=0,\\
x_{12}x_{23}-x_{23}x_{13}-x_{13}x_{12}&=0,\\
x_{23}x_{12}-x_{13}x_{23}-x_{12}x_{13}&=0.
\end{aligned}
$$

\begin{notation}[$C'$, $S$, $\Omega'_1$, $\Omega'_2$, $\Omega'_3$] \label{xxnot3.4}
Let $C'$ be the $\kk$-algebra generated by $x_{12}, x_{13}, x_{23}$, subject to the relations
$$\begin{aligned}
x_{12}^2+x_{13}^2&=0,\\
x_{12}x_{23}-x_{23}x_{13}-x_{13}x_{12}&=0,\\
x_{23}x_{12}-x_{13}x_{23}-x_{12}x_{13}&=0.
\end{aligned}
$$
Let $S$ be the $\kk$-algebra generated by $x_{12}$ and $x_{13}$, subject to the relation
$$x_{12}^2+x_{13}^2=0.$$
Moreover, put  $\Omega'_1:=x_{23}^2$, put $\Omega'_2:=x_{12}x_{23}x_{13}-x_{13}x_{23}x_{12}$
and put $\Omega'_3:=x_{13}^2$. 
\end{notation}

For the result below, we refer the reader to \cite[Chapter~2]{GW} for details 
about $\sigma$-derivations and Ore extensions.

\begin{lemma}
\label{xxlem3.5} 
Retain Notation~\ref{xxnot3.4}. Then the following 
statements hold.
\begin{enumerate}
\item[(1)]
$S$ is connected $\mathbb{N}$-graded, Noetherian, Auslander-regular, 
Cohen-Macaulay 
of global dimension two.

\smallskip

\item[(2)]
The map $\sigma: x_{12}\mapsto x_{13}$ and $x_{13}\mapsto x_{12}$ defines an algebra
automorphism of~$S$, and $\delta: x_{12}\mapsto x_{12}x_{13}$ and 
$x_{13}\mapsto -x_{13}x_{12}$ defines a $\sigma$-derivation of $S$.

\smallskip

\item[(3)]
$C'$ is an Ore extension $S[x_{23};\sigma,\delta]$.
Hence, $C'$ is a connected $\mathbb{N}$-graded, Noetherian, 
Auslander-regular, Cohen-Macaulay $\kk$-algebra of global dimension three.

\smallskip

\item[(4)]
$\{\Omega'_1,\Omega'_2,\Omega'_3\}$ is a normal sequence in $C'$.

\smallskip

\item[(5)]
$C'/(\Omega'_1,\Omega'_2,\Omega'_3)=\EE_3$.
\end{enumerate}
\end{lemma}

\begin{proof}
(1) This holds as  $S$ is isomorphic to the $(-1)$-skew polynomial ring
$\Bbbk_{-1}[z_1,z_2]$, 
which is well known to possess the desired properties.

\medskip

(2) It is clear that $\sigma$ is an algebra automorphism.
To check that $\delta$ is a $\sigma$-derivation, we calculate
$$\begin{aligned}
\delta (x_{12}^2+x_{13}^2) &=\delta(x_{12}) x_{12}+\sigma(x_{12})\delta(x_{12})
+\delta(x_{13}) x_{13}+\sigma(x_{13})\delta(x_{13})\\
&=(x_{12}x_{13})x_{12}+x_{13}(x_{12}x_{13})-(x_{13}x_{12})x_{13}-x_{12}(x_{13}x_{12}) \;= 0.
\end{aligned}
$$

\medskip

(3) By part (2) and the definition of $C'$, we see that 
$C'$ is an Ore extension $S[x_{23};\sigma,\delta]$. Now 
the second statement holds by (1) and several standard 
results including \cite[Theorem~4.2]{Ek} and \cite[page~184]{LS}.

\medskip

(4) First we claim that $\Omega'_1:=x_{23}^2$ is central (so, normal) in $C'$. 
We check:
$$\begin{aligned}
x_{23}^2 x_{12}&=x_{23} (x_{13}x_{23}+x_{12}x_{13})\\
&=(x_{12}x_{23}-x_{13}x_{12})x_{23}+(x_{13}x_{23}+x_{12}x_{13})x_{13}\\
&=x_{12}x_{23}^2-x_{13}(x_{12}x_{23}-x_{23}x_{13})+x_{12}x_{13}x_{13}\\
&=x_{12}x_{23}^2-x_{13}(x_{13}x_{12})+x_{12}x_{13}x_{13}\\
&=x_{12}x_{23}^2.
\end{aligned}
$$
Similarly, $x_{23}^2 x_{13}=x_{13}x_{23}^2$ (and $x_{23}^2$ 
commutes with $x_{23}$). So, the first claim holds.

\smallskip

Second we claim that $\Omega'_2$ is normal in $C'/(\Omega'_1)$.
We calculate
$$\begin{aligned}
\Omega'_2 x_{12}&=x_{12}x_{23}(x_{13}x_{12})-x_{13}(x_{23}x_{12})x_{12}\\
&=x_{12}x_{23}(x_{12}x_{23}-x_{23}x_{13})-x_{13}(x_{13}x_{23}+x_{12}x_{13})x_{12}\\
&=x_{12}(x_{23}x_{12})x_{23}-x_{13}^2x_{23}x_{12}-x_{13}x_{12}(x_{13}x_{12})\\
&=x_{12}(x_{13}x_{23}+x_{12}x_{13})x_{23}-x_{13}^2x_{23}x_{12}-x_{13}x_{12}
(x_{12}x_{23}-x_{23}x_{13})\\
&=x_{12}^2x_{13}x_{23}-x_{13}^2x_{23}x_{12}-x_{13}x_{12}
(x_{12}x_{23}-x_{23}x_{13})\\
&=x_{13}[-x_{13}^2x_{23}-x_{13}x_{23}x_{12}-x_{12}
(x_{12}x_{23}-x_{23}x_{13})]\\
&=x_{13}(-x_{13}x_{23}x_{12}+x_{12}x_{23}x_{13})\\
&=x_{13}\Omega'_2.
\end{aligned}
$$
Similarly, $\Omega'_2 x_{13}=x_{12}\Omega'_2$ and $\Omega'_2 x_{23} = -x_{23}\Omega'_2$. 
Thus $\Omega'_2$ is normal in $C'/(\Omega'_1)$.

\smallskip

Finally we claim that $\Omega'_3:=x_{13}^2$ is central  in $C'/(\Omega'_1,\Omega'_2)$. We calculate
$$\begin{aligned}
x_{13}^2 x_{12}&=x_{13}(x_{12}x_{23}-x_{23}x_{13})\\
&=(x_{12}x_{23}-x_{23}x_{13})x_{23}-(x_{23}x_{12}-x_{12}x_{13})x_{13}\\
&=-x_{23}x_{13}x_{23}-x_{23}x_{12}x_{13}+x_{12}x_{13}^2\\
&=-x_{23}(x_{23}x_{12}-x_{12}x_{13})-x_{23}x_{12}x_{13}+x_{12}x_{13}^2\\
&=x_{12}x_{13}^2
\end{aligned}
$$
and
$$\begin{aligned}
x_{23}x_{13}^2&=(x_{12}x_{23}-x_{13}x_{12})x_{13}\\
&=x_{12}x_{23}x_{13}-x_{13}(x_{23}x_{12}-x_{13}x_{23})\\
&=x_{13}^2 x_{23}.
\end{aligned}
$$
Since $x_{13}^2$ commutes with $x_{13}$, we have that $x_{13}^2$ is central 
in $C'/(\Omega'_1, \Omega'_2)$.

\medskip

(5) By comparing the generators and relations, one sees that
$C'/(\Omega'_1,\Omega'_3)=\EE_3$. It remains to show that $\Omega'_2=0$
in $C'/(\Omega'_1,\Omega'_3)$. We check, in $C'/(\Omega'_1,\Omega'_3)$ that:
$$\Omega'_2  ~=~ x_{12}x_{23}x_{13}-x_{13}x_{12}x_{13}-x_{13}x_{23}x_{12}+x_{13}x_{12}x_{13}~ =~x_{23}x_{13}^2-x_{13}^2x_{23} ~= 0.$$
Thus, $\Omega'_2=0$ in $C'/(\Omega'_1,\Omega'_3)$ as desired.
\end{proof}

\begin{theorem}
\label{xxthm3.6} 
We obtain that $\EE_3$ is a cci, and that $cci(\EE_3)=3$. 
As a consequence, $\EE_3$ is Auslander-Gorenstein, Cohen-Macaulay, and Frobenius.
\end{theorem}

\begin{proof}
Recall Notation~\ref{xxnot3.4}. By Lemma~\ref{xxlem3.5}(3,4,5), we 
get that $\{\Omega'_1,\Omega'_2,\Omega'_3\}$ is a
normal sequence of an Auslander regular and Cohen-Macaulay
algebra $C'$ such that $\EE_3 =  C'/( \Omega'_1,\Omega'_2,\Omega'_3)$.
By the definition of $C'$ in Lemma~\ref{xxlem3.5}(3), one sees that
$\GKdim C'=3$ [Remark~\ref{xxrem1.2}(2)]. Now by \cite[(2.8)]{FK}, 
$\GKdim \EE_3=0.$ Hence, $\GKdim C'/( \Omega'_1,\Omega'_2,\Omega'_3) = 0$ 
and with Proposition~\ref{xxpro3.3}(2) we obtain that 
$\{\Omega'_1,\Omega'_2,\Omega'_3\}$ is a regular normal sequence of
$C'$. By definition, $\EE_3$ is a cci.

\smallskip Since $\EE_3$ is finite dimensional and a cci, it is
Auslander Gorenstein and Cohen-Macaulay, consequently, Frobenius.

\smallskip For the cci number of $\EE_3$, note that by the argument above we 
have $cci(\EE_3)\leq 3$. If $cci(\EE_3)\leq 2$, then there is a 
Noetherian AS-regular algebra $B$  
and a regular normal sequence $\{f_1,f_2\}$ of $B$ 
such that $\EE_3\cong B/(f_1,f_2)$.
Recall that $\EE_3$ is Auslander-Gorenstein and Cohen-Macaulay, so
by \cite[Theorem~3.6 and Remark~5.10]{Lev}
$B$ is Auslander-regular and Cohen-Macaulay of 
GK-dimension and global dimension two. Such a $B$
must be generated by two elements \cite[Introduction]{AS}. This contradicts
the fact that $\EE_3$ is generated by three elements.
Thus, $cci(\EE_3)\geq 3$.
\end{proof}

Now we consider the quadratic dual $\EE_3^{!}$ of $\EE_3$.
Recall that $\EE_3^{!}$ is generated by $y_{12}$,~$y_{13}$,~$y_{23}$ 
(where we suppress the comma in the subscript of indices) 
subject to the relations
\begin{align*}
y_{12}y_{23}+y_{23}y_{13} ~=~ y_{12}y_{23}+y_{13}y_{12}&=0,\\
y_{23}y_{12}+y_{13}y_{23} ~=~ y_{23}y_{12}+y_{12}y_{13}&=0.
\end{align*}

\begin{definitionlemma}[$a$, $b$, $c$, $C''$] \label{xxdeflem3.7}
Let $a:=y_{13}+y_{23}$, $b:=y_{13}-y_{23}$, $c:=y_{12}$. After a linear transformation, 
$\EE_3^{!}$ is generated by $a$, $b$, $c$, subject to the following relations, 
$$ 
ca+ac ~=~ cb-bc ~=~
-2bc+a^2-b^2 ~=~
-2ac+(ab-ba) ~=0.
$$
Moreover, let $C''$ be the algebra generated by $a$, $b$, $c$,
subject to the relations
$$
ca+ac ~=~
cb-bc ~=~
a^2b-ba^2 ~=~
ab^2-b^2a ~=0. 
$$ 

\vspace{-.25in}
\qed
\end{definitionlemma}

\begin{lemma}
\label{xxlem3.8} The algebra $C''$ is Noetherian, 
AS-regular, Auslander-regular,
Cohen-Macaulay, of global dimension four, 
of GK-dimension four, and it has Hilbert series $(1-t)^{-3}(1-t^2)^{-1}$.
\end{lemma}

\begin{proof} Let $B$ be the $\kk$-algebra
$\Bbbk\langle a,b\rangle/(a^2b-ba^2,ab^2-b^2a)$. This algebra   
is Noetherian  AS-regular algebra of global dimension three 
\cite[(8.5)]{AS}, with Hilbert series $(1-t)^{-2}(1-t^2)^{-1}$ 
\cite[(1.15)]{AS}, and is also  Auslander-regular and 
Cohen-Macaulay (e.g., via \cite[Corollary~5.10]{Lev}) and has 
GK-dimension three (e.g., via Proposition~\ref{xxpro3.3}(2)). 
Note that $C''$ is an Ore extension
$B[c,\sigma]$ where $\sigma: a\mapsto -a,~ b\mapsto b$. So, 
$C''$ has Hilbert series $H_{B}(t)/(1-t^{\deg c}) = (1-t)^{-3}(1-t^2)^{-1}$.
The rest of the result follows from several standard results 
including \cite[Theorem~4.2]{Ek} and \cite[page~184]{LS}.
\end{proof}

\begin{lemma}[$\Omega''_1$, $\Omega''_2$, $\Omega''_3$]
\label{xxlem3.9} Retain the  notation of Definition-Lemma~\ref{xxdeflem3.7}.
\begin{enumerate}
\item[(1)]
Let $\Omega''_1:=(ab+ba)c$, 
$\Omega''_2:=-2bc+a^2-b^2$, and 
$\Omega''_3:=-2ac+(ab-ba)$.
Then, $\{\Omega''_1,\Omega''_2,\Omega''_3\}$ is a 
normal sequence in $C''$.

\smallskip

\item[(2)]
$C''/(\Omega''_1,\Omega''_2,\Omega''_3)=\EE_{3}^{!}$.
\end{enumerate}
\end{lemma}

\begin{proof}
(1) 
Note that $(ab+ba)c=-c(ab+ba)$ in $C''$. 
It is also easy to
check that $(ab+ba)$ is central  in the $AS$-regular algebra
$B$ from the proof of Lemma~\ref{xxlem3.8}. Since $C''$ is an 
Ore extension $B[c;\sigma]$ (from the proof of Lemma~\ref{xxlem3.8}), 
$(ab+ba)$ is normal in $C''$, and so is  $c$.
Hence, $\Omega''_1=(ab+ba)c$ is normal in $C''$. 

\smallskip To show $\Omega''_2$
is normal in $C''/(\Omega''_1)$, note that 
$[\Omega''_2,c]=[\Omega''_2,b]=0$ and 
$$
[\Omega''_2, a]=\Omega''_2 a-a\Omega''_2=2(ab+ba)c=0
$$
in $C''/(\Omega''_1)$. 

\smallskip Now to see that $\Omega_3''$ in normal in $C''/(\Omega''_1, \Omega''_2)$, note
$\Omega''_3 c+c\Omega''_3=
\Omega''_3 a+a\Omega''_3
=0$. Lastly,
we have
$$
\Omega''_3 b+b\Omega''_3~ =~(-2ac+(ab-ba))b+b(-2ac+(ab-ba))~ = ~ -2acb-2bac=0\\
$$
in $C''/(\Omega''_1,\Omega''_2)$.
Therefore $\{\Omega''_1,\Omega''_2,\Omega''_3\}$ is a normal
sequence in $C''$.

\medskip 

(2) In the algebra
$\EE_3^! \cong \kk\langle a,b,c\rangle/(ca+ac,~cb-bc,~\Omega''_2,~\Omega''_3)$
we get that
$$[a^2,b] = \Omega''_2 b - b \Omega''_2 \quad \text{ and } \quad [a,b^2] = \Omega''_3 b + b \Omega''_3,$$
and in the algebra
$\kk\langle a,b,c\rangle/(ca+ac,~cb-bc,~[a^2,b],~ [a,b^2],~ \Omega''_2,~\Omega''_3)$ we get that
$$\Omega''_1 = abc-bca ~= \textstyle\frac{1}{2}a(a^2-b^2) - \frac{1}{2}(a^2-b^2)a ~= \frac{1}{2}[b^2,a] = 0,
$$
as required.
\end{proof}

This brings us to our main result about $\EE_3^!$.

\begin{theorem}
\label{xxthm3.10} 
\begin{enumerate}
\item[(1)] $\EE_3^{!}$ is a cci and 
$H_{\EE_{3}^{!}}(t)=\frac{(1+t)(1+t+t^2)}{(1-t)}$.

\smallskip

\item[(2)] $\EE_{3}^{!}$ is Artin-Schelter-Gorenstein and
Auslander-Gorenstein.

\smallskip

\item[(3)] $\EE_{3}^{!}$ is not Artin-Schelter-regular nor Auslander-regular.
\end{enumerate}
\end{theorem}

\begin{proof}
(1) By Lemmas~\ref{xxlem3.8} and~\ref{xxlem3.9},
there exists a normalizing sequence $\{\Omega''_1,\Omega''_2,\Omega''_3\}$ 
 of an Auslander-regular, Cohen-Macaulay
algebra $C''$ with $\EE_3^! \cong C''/( \Omega''_1,\Omega''_2,\Omega''_3)$.
Moreover, $\GKdim C''=4$ by Lemma~\ref{xxlem3.8}. Now 
consider the normal element $\Omega''_4:=\frac{1}{2}(a^2+b^2)+c^2$ 
of $\EE_3^!$. Then with variable ordering $a<b<c$, one can compute 
(via the GBNP package of GAP \cite{GAP})  that the ideal of relations for 
$\EE_3^!/(\Omega''_4)$ has Gr\"{o}bner basis 
\[
\left\{
\begin{array}{lllll}
\smallskip

\textstyle ba+2ac-ab, ~&bc+\frac{1}{2}(b^2-a^2), ~&ca+ac, ~&cb-bc, ~&c^2+\frac{1}{2}(b^2+a^2),\\
a^3, &b^3+\frac{1}{3}(2a^2c+a^2b)&&&
\end{array}
\right\}.
\] 
Thus, $C''/( \Omega''_1,\Omega''_2,\Omega''_3, \Omega''_4) = \EE_3^!/(\Omega''_4)$ has Hilbert series $1+3t+4t^2+3t^3+t^4$ and has GK-dimension 0. Since $\EE_3^!= C''/( \Omega''_1,\Omega''_2,\Omega''_3)$ has GK-dimension 1 [Theorem~\ref{xxthm0.3}(2)], we obtain that $\{\Omega''_1,\Omega''_2,\Omega''_3\}$ is a regular normal sequence of
$C''$  by Proposition~\ref{xxpro3.3}(2). By definition, $\EE_3^{!}$ is a cci.

\smallskip Since $\{\Omega''_1,\Omega''_2,\Omega''_3\}$ is a regular normal sequence of
$C''$ with degree 3, 2, and 2 respectively, we get by Lemma~\ref{xxlem3.8} that
$$H_{\EE_{3}^{!}}(t)=H_{C''}(t) (1-t^2)^2(1-t^3)  =  \frac{(1-t^2)^2(1-t^3)}{(1-t)^{3} (1-t^2)}  =  \frac{(1+t)(1+t+t^2)}{1-t}.$$

\medskip

(2) This follows from (1), along with Lemma~\ref{xxlem3.8} and \cite[Theorems~3.6 and~6.3]{Lev}.

\medskip

(3) We have by Theorem~\ref{xxthm0.3}(1) and 
\cite[Corollary~1.2]{SZ} that if $\EE_3^!$ 
had finite global dimension, then $\EE_3^!$ 
would be a domain. But this contradicts 
Theorem~\ref{xxthm0.3}(3). So, $\EE_3^!$ 
is neither AS-regular nor Auslander-regular.
 \end{proof}

Now we turn our attention to the AS-Cohen-Macaulay 
property of $\EE_n^!$, and for this we need result 
on the depth of  $\EE_n^!$. We start with the 
following preliminary result.

\begin{lemma}
\label{xxlem3.11} 
For connected $\mathbb{N}$-graded algebras $A$ and $C$, we have the 
statements below.
\begin{enumerate}
\item[(1)] Suppose that $A$ has a regular homogeneous
element $f$ of positive degree. Then, $\depth A>0$. 

\smallskip

\item[(2)] Let $C$ be a Noetherian commutative algebra and 
$M$ be a finitely generated module over $C$. If $\depth M>0$, 
then there is a homogeneous element of positive degree $f\in C$ 
that is regular on $M$.

\smallskip

\item[(3)] Let $A$ be Noetherian and finitely generated over
its affine center.
Suppose $A$ has a homogeneous element $g$ 
of positive degree such that
$\GKdim A g\leq r$. Then $\depth A\leq r$.
\end{enumerate}
\end{lemma}

\begin{proof}
(1) This is a connected $\mathbb{N}$-graded analogue of 
a standard homological result from the theory of 
commutative local rings, see e.g., \cite[Proposition~1.2.3]{BH}.
Suppose, by way of contradiction, 
that $\depth A=0$. Then, by definition, $\Hom_A(\kk, A) \neq 0$ 
and so there exists a 1-dimensional nonzero ideal $I$ of $A$. 
Let $x$ be a generator of $I$. Since $f$ has positive degree, 
we get that $fx=0$. This contradicts the regularity of $f$. 

\medskip

(2) Let $\mathfrak{m}:= C_{\geq 1}$ be the maximal graded ideal 
of $C$. If $\mathfrak{m}$ consists of non-regular elements of 
$M$, then $\mathfrak{m}$ is contained in the union of the 
associated primes of the $C$-module $M$. By the Noetherian 
property and prime avoidance, $\mathfrak{m}$ is actually 
contained in one associated prime $\mathfrak{p}$ of $M$. 
Thus $\mathfrak{p}=\mathfrak{m}$. Now there exists a 
monomorphism $C/\mathfrak{p} \to M$. 
Composing with the natural isomorphism 
$C/\mathfrak{m} \to C/\mathfrak{p}$, we get a nonzero 
$C$-module map $C/\mathfrak{m} \to M$. Thus, 
$\Hom_C(\kk, M) \neq 0$ and $\depth M = 0$.

\medskip

(3)
If $\depth A = 0$, then we are done. Now suppose $\depth A >0$.
Let $C$ be the center of $A$. Since $A$ is finitely generated
over $C$, $\depth_{C} A=\depth_{A} A>0$.
By part (2), there exists a homogeneous element $f\in C$ of 
positive degree that is regular on $A$. Replacing $f$ by 
$f^n$ for some $n\gg 0$, we may assume that $\deg f > \deg g$.
Consider the sequence
$$0 \longrightarrow Ag \overset{\cdot f}{\longrightarrow} 
Ag \longrightarrow Ag/Agf \longrightarrow 0.$$
Since $\GKdim(Ag) \leq  r$ by assumption, we have that 
$\GKdim(Ag/Agf) \leq r-1$ \cite[Proposition~8.3.5]{MR}.
Define $\overline{A}:=A/(f)$ and let $\overline{g}$ be 
the image of $g$ in $\overline{A}$. Using the surjection
$Ag/Agf \twoheadrightarrow Ag/(Ag \cap Af)$, we have 
$$\GKdim(\overline{A}\overline{g}) = \GKdim(Ag/(Ag \cap Af)) \leq r-1.$$
By induction, $\depth(\overline{A}) \leq r-1$. 
Now the result holds by Rees' lemma \cite[Theorem~8.34]{Rot}.
\end{proof}

\begin{theorem} 
\label{xxthm3.12}
For every $n\geq 2$, $\depth (\EE_n^!)\leq 1$. 
As a consequence, for $n \geq 4$, we get that 
$\depth \EE_n^! < \GKdim \EE_n^!$.
\end{theorem}

\begin{proof}
Let 
$$g=\prod_{i<j} a_{i,j}$$
which is nonzero by Corollary \ref{xxcor2.11}. It is easy to check that
$a_{i,j} g=a_{1,2} g$ by \eqref{E1.8.1}. Hence $\DD_n  g= 
\Bbbk[a_{1,2}] g$ or $\GKdim \DD_n g=1$.
Since $\EE_{n}^{!}$ is finitely generated over $\DD_n$, $\GKdim \EE_{n}^! g=1$.
The assertion follows by Lemma \ref{xxlem3.11}(3) 
and the consequence holds by Theorem~\ref{xxthm0.3}(2).
\end{proof}

\medskip

\begin{proof}[Proof of Theorem~\ref{xxthm0.6}] 
Refer to Figure~\ref{fig:hom} throughout.  By 
Theorem~\ref{xxthm0.3}(2), we have that to establish parts (1)-(3) 
it suffices to show that (a) $\EE_2^!$ is AS-regular, 
(b) $\EE_3^!$ is AS-Gorenstein and but not AS-regular, 
and (c) $\EE_n^!$ is not AS-Cohen-Macaulay for $n \geq 4$. 
Now (a) holds as $\EE_2^!$ is the commutative polynomial 
ring $\kk[y_{1,2}]$; (b) holds by Theorem~\ref{xxthm3.10}(2,3); 
and (c) follows from Theorem~\ref{xxthm3.12} and 
Theorem~\ref{xxthm0.3}(1).

\smallskip Moreover, (4) holds since $\EE_2^! = \kk[y_{1,2}]$ (AS-regular), 
since $\EE_2 = \kk[x_{1,2}]/(x_{1,2}^2)$  (quotient of an 
AS-regular algebra by a regular element), and by 
Theorems~\ref{xxthm3.10}(1) and~\ref{xxthm3.6} for 
the algebras $\EE_3^!$ and $\EE_3$, respectively.
\end{proof}

\begin{proof}[Proof of  Corollary~\ref{xxcor0.8}] 
Refer to Figure~\ref{fig:hom} throughout.
The result on Auslander-regularity holds for $n=2$ as 
$\EE_2^! = \kk[y_{1,2}]$ is clearly Auslander-regular, 
for $n=3$ by Theorem~\ref{xxthm3.10}(3), and for 
$n \geq 4$ by~Theorem~\ref{xxthm0.6}(3). The result 
on Auslander-Gorenstein holds for $n=2$ by 
Auslander-regularity, for $n=3$ by  
Theorem~\ref{xxthm3.10}(2),  and for $n\geq4$ by 
Theorem~\ref{xxthm0.6}(3). The result on the 
Cohen-Macaulay condition holds for $n=2,3$ by the 
Auslander-Gorenstein condition with 
Theorem~\ref{xxthm0.3}(1), and for $n\geq4$ by 
Theorems~\ref{xxthm0.6}(3) and~\ref{xxthm0.3}(1).
\end{proof}


\medskip

\section{Further directions}
\label{xxsec4}

First we make the following remark about Question~\ref{xxque0.4} and Conjecture~\ref{xxcon0.7} discussed in the Introduction.

\begin{remark} 
\label{xxrem4.1}
If $\EE_{n}^!$ is semiprime (i.e., if Question~\ref{xxque0.4}(1) 
is affirmative), then we obtain that $\depth \EE_{n}^{!}=1$ 
(i.e., that Conjecture \ref{xxcon0.7} holds). Namely, for any 
connected graded algebra $A\neq \Bbbk$, $\depth A=0$ implies that
$A$ is not semiprime as $\Hom_A(\Bbbk, A)\subseteq A$ is a nonzero 
nilpotent ideal of $A$. Therefore $\depth \EE_{n}^!\geq 1$ when 
${\mathcal E}_n^!$ is semiprime. Now Conjecture \ref{xxcon0.7} 
follows from Theorem \ref{xxthm3.12}.
\end{remark}

In addition to Question~\ref{xxque0.4} and Conjecture~\ref{xxcon0.7}, 
along with Question~\ref{xxque2.7}, we present here three other 
suggestions for further study of the quadratic dual $\EE_n^!$ of 
Fomin-Kirillov algebras, yet there are numerous other directions 
that one could pursue motivated by Fomin-Kirillov's work \cite{FK} 
alone.

\subsection{On the center of $\EE_n^!$} 
\label{xxsec4.1}
In Section~\ref{xxsec1} we introduced a commutative subalgebras $\CC_n$ 
and ${\mathcal D}_n$ of $\EE_n^!$ in order to prove 
Theorem~\ref{xxthm0.3}; recall ${\mathcal C}_n \cong {\mathcal D}_n$ 
by Proposition~\ref{xxpro2.10}. But we ask:

\begin{question} 
\label{xxque4.2} 
What is the presentation of the center $Z(\EE_n^!)$ of $\EE_n^!$?
\end{question}

On a related note, there is an important subalgebra of the 
Fomin-Kirillov algebra $\EE_n$ constructed in \cite{FK} 
generated by its {\it Dunkl elements}. This subalgebra, 
which we denote by $\mathcal{F}_n$, is isomorphic to the 
cohomology of a flag manifold \cite[Theorem~7.1]{FK}; 
the full presentation of $\mathcal{F}_n$ is also established 
in that result.

\begin{question} 
\label{xxque4.3} 
What is the connection between $\mathcal{F}_n$ and the 
commutative algebras $Z(\EE_n^!)$ and $\CC_n$ 
discussed above?
\end{question}

\subsection{On the $S_n$-action on $\EE_n^!$ and related algebras}
\label{xxsec4.2} 
As discussed in Milinski and Schneider's study of 
{\it Nichols algebras} over Coxeter groups \cite{MS}, the 
Fomin-Kirillov algebra $\EE_n$ admits an action of the 
symmetric group $S_n$, and moreover, can be realized as a 
braided Hopf algebra in the category ${}^{S_n}_{S_n}\mathcal{YD}$ 
of Yetter-Drinfeld modules over $S_n$. Namely, $\EE_n$ arises 
as a {\it pre-Nichols algebra} in ${}^{S_n}_{S_n}\mathcal{YD}$, 
and in fact, it is conjectured that $\EE_n$ is an honest Nichols 
algebra in ${}^{S_n}_{S_n}\mathcal{YD}$ (which has been verified 
for $n\leq 5$). Two algebras that are of interest in this context 
are the invariant subalgebra $\EE_n^{S_n}$ and the skew group 
algebra $\EE_n \rtimes S_n$. For instance, for a finite group 
$G$, there is a useful functor from ${}^{G}_{G}\mathcal{YD}$ 
to the category of $\kk$-vector spaces sending a braided 
Hopf algebra $\mathcal{B}$ to the $\kk$-Hopf algebra 
$\mathcal{B} \rtimes G$. Vital classes of finite-dimensional 
pointed Hopf algebras have been constructed in this fashion.

\smallskip Now the quadratic dual $\EE_n^!$ also admits an action of 
the symmetric group $S_n$,
and an interesting direction of further research is to  
study the behavior of the resulting invariant ring and skew group algebra.

\subsection{\bf On the Koszulity of $\EE_n^!$} 
\label{xxsec4.3}
As mentioned in the Introduction, the Fomin-Kirillov algebras 
$\EE_n$ fail to be Koszul for $n\geq 3$, due to a result of 
Roos \cite{Ro}; so the same result holds for the quadratic 
dual $\EE_n^!$. Towards understanding  the cohomology rings 
$\Ext_{\EE_n}^*(\kk, \kk)$ and $\Ext_{\EE_n^!}^*(\kk, \kk)$ 
(for which  $\EE_n^!$ and $\EE_n$,  respectively, is the 
subalgebra generated in degree one), the failure of Koszulity 
should be studied more carefully. Indeed, if $\EE_n^!$ and 
$\EE_n$ were Koszul, then they would equal $\Ext_{\EE_n}^*(\kk, \kk)$ 
and $\Ext_{\EE_n^!}^*(\kk, \kk)$, respectively. 

As in \cite[Section~2.4]{PP}, we say that a graded algebra 
$A$ is {\it $p$-Koszul} if 
$$\Ext^{i,j}_A(\kk,\kk) = 0, \quad \forall i<j\leq p.$$
For example, any graded algebra is 1-Koszul, any graded 
algebra generated in degree one is 2-Koszul, and any 
quadratic algebra is 3-Koszul. Moreover, a graded (quadratic) 
algebra is Koszul if and only if it is $p$-Koszul for all 
$p \geq 1$. By  \cite[Proposition~2.4.5]{PP}, if $A$ is a 
$(p-1)$-Koszul quadratic algebra, then for each $2< i <p$ 
there is a natural perfect pairing 
$$\Ext_A^{i,p}(\kk,\kk) \otimes \Ext_{A^!}^{p-i+2,p}(\kk,\kk) 
\longrightarrow \kk.$$
So, we ask:

\begin{question} 
\label{xxque4.4} 
What is the maximum value of $p = p(n)$ for which $\EE_n^!$ is $p$-Koszul?
\end{question}


\bigskip

\subsection*{Acknowledgments}
We would like to thank Pavel Etingof for many 
useful conversations on this subject and for 
generously sharing his ideas and 
suggesting to include his results in this paper; 
this is the subject of Section~\ref{xxsec2}. 
We also like to thank Mohamed Omar for useful 
discussions on the combinatorics of Section~\ref{xxsec2}.

\smallskip C. Walton was partially supported by a research 
fellowship from the Alfred P. Sloan foundation, 
and by the US National Science Foundation grant \#DMS-1663775.  
J.J. Zhang was partially supported by the US National 
Science Foundation \#DMS-1700825. 
This work was completed during C. Walton's visits to 
the University of Washington-Seattle and during the 
authors attendance at the ``Quantum Homogeneous Spaces" 
workshop at the International Centre for Mathematical 
Sciences in Edinburgh; we appreciate the institutions' 
staff for their hospitality and assistance during these 
stays.
 
 \bigskip

\providecommand{\bysame}{\leavevmode\hbox to3em{\hrulefill}\thinspace}
\providecommand{\MR}{\relax\ifhmode\unskip\space\fi MR }
\providecommand{\MRhref}[2]{%

\href{http://www.ams.org/mathscinet-getitem?mr=#1}{#2} }
\providecommand{\href}[2]{#2}


\begin{thebibliography}{9999999}



\bibitem{AG}
N. Andruskiewitsch and M. Gra\~{n}a, 
From racks to pointed Hopf algebras, 
{\it Adv. Math.} {\bf 178} (2003), no. 2, 177--243.

\bibitem{AS}
 M. Artin and W. F. Schelter, 
 Graded algebras of global dimension 3, 
 {\it Adv. in Math.} {\bf 66} (1987), 171--216.

\bibitem{BEER}
L. Bartholdi, B. Enriquez, P. Etingof and R. Rains, 
Groups and Lie algebras corresponding to the Yang-Baxter equations,
{\it J. Algebra} {\bf 305} (2006), no. 2, 742--764. 




\bibitem{Bas}
H. Bass,
On the ubiquity of Gorenstein rings, 
{\it Math. Z.} {\bf 82} (1963), 8--28.

\bibitem{Baz}
Y. Bazlov, 
Nichols-Woronowicz algebra model for Schubert calculus on Coxeter groups,
{\it J. Algebra} {\bf 297} (2006), no. 2, 372--399.


\bibitem{Be}
G.M. Bergman, 
The diamond lemma for ring theory, 
{\it Adv. in Math.} {\bf 29} (1978), no. 2, 178--218.




\bibitem{BLM}
J. Blasiak, R. I. Liu, K. M\'{e}sz\'{a}ros, Karola, 
Subalgebras of the Fomin-Kirillov algebra,
{\it  J. Algebraic Combin.} {\bf 44} (2016), no. 3, 785--829.

\bibitem{BG}
K. A. Brown and K. R. Goodearl, 
Lectures on algebraic quantum groups,
{\it Advanced Courses in Mathematics}, 
CRM Barcelona. Birkh\"{a}user Verlag, Basel, Switzerland, 2002.

\bibitem{BH}
W. Bruns and J. Herzog, 
Cohen-Macaulay rings, 
{\it Cambridge Studies in Advanced Mathematics},  39 Cambridge University Press, Cambridge, 1993.

\bibitem{CV}
T. Cassidy and M. Vancliff,
Generalizations of graded Clifford algebras and of complete intersections,
{\it J. Lond. Math. Soc.} (2) {\bf 81} (2010), no. 1, 91--112.

\bibitem{GAP}
A. Cohen and J. Knopper,
GAP package GBNP: computing Gr\"{o}bner bases of noncommutative polynomials,
{\tt http://www.gap-system.org/Packages/gbnp.html}.

\bibitem{Ek} 
E. K. Ekstr\"{o}m,
The Auslander condition on graded and filtered Noetherian rings, 
in {\it Sem. d'Algebre P. Dubreil et M.-P. Malliavin} 1987-88 (M.-P. Malliavin, ed.), 
{\it Lecture Notes in Math.}, 220--245, 1404, Springer-Verlag, Berlin, Germany, 1989.

\bibitem{ETW}
J.S. Ellenberg, T. Tran, and C. Westerlan,
Fox-Neuwirth-Fuks cells, quantum shuffle algebras,
and Malle's conjecture for function fields (preprint), 
\url{https://arxiv.org/pdf/1701.04541.pdf}.

\bibitem{FK}
S. Fomin and A.N. Kirillov,
Quadratic algebras, Dunkl elements, and Schubert calculus,
{\it Advances in geometry}, 147--182, Progr. Math., 172, 
Birkh{\"a}user Boston, Boston, MA, USA, 1999.


\bibitem{FP}
S. Fomin and C. Procesi, 
Fibered quadratic Hopf algebras related to Schubert calculus,  
{\it J. Algebra} {\bf 230} (2000), no. 1, 174--183.

\bibitem{GR}
I. Gelfand and V. Retakh,
Quasideterminants I,
{\it Selecta Math.} (N.S.) {\bf 3} (1997), no. 4, 517--546. 


\bibitem{GW}
K. R. Goodearl and R. B. Warfield, Jr.,
An introduction to noncommutative Noetherian rings, Second edition,
{\it London Mathematical Society Student Texts},  61, 
Cambridge University Press, Cambridge, UK 2004.


\bibitem{Jo1}
P. J{\o}rgensen,
Local cohomology for non-commutative graded algebras,
\emph{Comm. Algebra} {\bf 25} (1997), 575--591. 


\bibitem{Jo2}
P. J{\o}rgensen, 
Properties of AS-Cohen-Macaulay algebras, 
\emph{J. Pure Appl. Algebra} {\bf 138} (1999), no. 2, 239--249. 

\bibitem{JZ}
P. J{\o}rgensen and J. J. Zhang, 
Gourmet's guide to Gorensteinness, 
{\it Adv. Math.} {\bf 151} (2000), no. 2, 313--345.

\bibitem{Karp}
G. Karpilovsky, Group representations. Vol. 2. 
{\it North-Holland Mathematics Studies}, 177. North-Holland Publishing Co., Amsterdam, 1993.


\bibitem{KM}
A. N. Kirillov and T. Maeno,
Extended quadratic algebra and a model of the equivariant cohomology ring of flag varieties,
{\it Algebra i Analiz} {\bf 22} (2010), no. 3, 155--176; 
translation in {\it St. Petersburg Math. J.} {\bf 22} (2011), no. 3, 447--462.

\bibitem{KKZ}
E. Kirkman, J. Kuzmanovich, and J.J. Zhang,
Noncommutative complete intersections,  
{\it J. Algebra} {\bf 429} (2015), 253--286.

\bibitem{KL}
G. R. Krause and T. H. Lenagan,  
Growth of algebras and Gelfand-Kirillov dimension. Revised edition. 
{\it Graduate Studies in Mathematics}, 22. American Mathematical Society, Providence, RI, 2000. 

\bibitem{La}
R. Laugwitz, 
On Fomin-Kirillov algebras for complex reflection groups, 
{\it Comm. Algebra} {\bf 45} (2017), no. 8, 3653--3666. 

\bibitem{Len}
C. Lenart, 
The K-theory of the flag variety and the Fomin-Kirillov quadratic algebra,
{\it J. Algebra} {\bf 285} (2005), no. 1, 120--135.

\bibitem{Lev}
T. Levasseur, 
Some properties of non-commutative regular graded rings, 
{\it Glasgow Math. J.} {\bf 34} (1992), 277--300.

\bibitem{LS}
T. Levasseur and J. T. Stafford, 
The quantum coordinate ring of the special linear group,
{\it J. Pure Applied Algebra} {\bf 86} (1993), 181--186.

\bibitem{LPWZ}
D.-M Lu, J.H. Palmieri, Q.-S. Wu, and J.J. Zhang,
Koszul equivalences in $A_\infty$-algebras,
{\it New York J. Math.} {\bf 14} (2008), 325--378. 


\bibitem{Ma1}
S. Majid,
Noncommutative differentials and Yang-Mills on permutation groups $S_n$, 
{\it Hopf algebras in noncommutative geometry and physics}, 189--213, 
Lecture Notes in Pure and Appl. Math., 239, Dekker, New York, 2005. 


\bibitem{Ma2}
S. Majid, 
Hodge star as braided Fourier transform,
{\it Algebr. Represent. Theory} {\bf 20} (2017), no. 3, 695--733. 


\bibitem{MT}
S. Majid and I. Toma\c{s}i\'{c},
On braided zeta functions,
{\it Bull. Math. Sci.} {\bf 1} (2011), no. 2, 379--396.

\bibitem{MR}
J. C. McConnell and J. C. Robson,
Noncommutative Noetherian rings, 
With the cooperation of L. W. Small. Revised edition. 
{\it Graduate Studies in Mathematics}, 30. American Mathematical Society, Providence, RI, USA, 2001.



\bibitem{MPP} 
K. M\'{e}sz\'{a}ros, G. Panova, and A. Postnikov, 
Schur times Schubert via the Fomin-Kirillov algebra,
{\it Electron. J. Combin.} {\bf 21} (2014), no. 1, Paper 1.39, 22 pp. 

\bibitem{MS}
A. Milinski and H. Schneider, 
Pointed indecomposable Hopf algebras over Coxeter groups. 
{\it New trends in Hopf algebra theory} (La Falda, 1999), 215--236, 
Contemp. Math., 267, Amer. Math. Soc., Providence, RI, 2000.


\bibitem{PP}
A. Polishchuk and L. Positselski, 
Quadratic algebras,
{\it University Lecture Series}, 37. American Mathematical Society, Providence, RI, USA 2005. 

\bibitem{Po}
A. Postnikov, 
On a quantum version of Pieri's formula. 
{\it Advances in geometry}, 371--383, 
Progr. Math., 172, Birkhäuser Boston, Boston, MA, 1999.

\bibitem{Ro}
J.-E. Roos,
Some non-Koszul algebras, 
\emph{Advances in geometry}, 385--389, 
Progr. Math., 172, Birkh{\"a}user Boston, Boston, MA, USA, 1999.

\bibitem{Rot}
J. J. Rotman, 
An Introduction to Homological Algebra. 
Springer Science \& Business Media, New York, NY, USA, 2008. 

\bibitem{Row}
L. Rowen, 
Ring theory, Vol. II,
{\it Pure and Applied Mathematics}, 128. 
Academic Press, Inc., Boston, MA, USA,1988.

\bibitem{SZ}
J. T. Stafford and J. J. Zhang,
Homological Properties of (Graded) Noetherian PI Rings,
{\it J. Algebra} {\bf 168} (1994,) no. 3,  988--1026.


\bibitem{SV}
D. \c{S}tefan and C. Vay, 
The cohomology ring of the 12-dimensional
Fomin-Kirillov algebra,
\emph{Adv. Math.} {\bf 291} (2016), 584--620.

\bibitem{vdB}
M. Van den Bergh, 
Existence theorems for dualizing complexes over non-commutative graded and
filtered rings, 
{\it J. Algebra} {\bf 195} (1997) 662--679. 

\bibitem{Zh}
J.J. Zhang,
Twisted graded algebras and equivalences of graded categories. 
{\it Proc. London Math. Soc.} (3) {\bf 72} (1996), no. 2, 281--311.


\end{thebibliography}
\end{document}